\documentclass[11pt]{article}

\usepackage[utf8]{inputenc} % allow utf-8 input
\usepackage[T1]{fontenc}    % use 8-bit T1 fonts
\usepackage[hyperfootnotes=false]{hyperref}       % hyperlinks (Alnur added hyperfootnotes=false to kill footnotes on footnotes)
\usepackage{url}            % simple URL typesetting
\usepackage{booktabs}       % professional-quality tables
\usepackage{amsfonts}       % blackboard math symbols
\usepackage{nicefrac}       % compact symbols for 1/2, etc.
\usepackage{microtype}      % microtypography
\usepackage{setspace}
\usepackage[round]{natbib}
\usepackage[left=1.25in,top=1.25in,right=1.25in,bottom=1.25in,head=1.25in]{geometry}

\usepackage{enumerate}
\usepackage{titling}

\newcommand{\ones}{\mathbf 1}
\newcommand{\reals}{{\mbox{\bf R}}}
\newcommand{\symm}{{\mbox{\bf S}}}  % symmetric matrices

\newcommand{\diag}{\mathop{\bf diag}}

\newcommand{\co}{{\mathop {\bf co}}} % Alnur
\newcommand{\dist}{\mathop{\bf dist{}}}
\newcommand{\argmin}{\mathop{\rm argmin}}

\newcommand{\eg}{{\it e.g.}}
\newcommand{\ie}{{\it i.e.}}

\usepackage{graphicx,psfrag,amsmath,amsfonts,verbatim} % cut, b/c throws uai margins off: fullpage,
\usepackage[small,bf]{caption}
\usepackage{url}

\ifdefined\aaslides
\else

\usepackage{amsthm}
\newtheorem{theorem}{Theorem}[section]
\newtheorem{lemma}[theorem]{Lemma}

\fi

\newcommand{\minimize}{\mathop{\mbox{minimize}}} % Ryan changed this
\newcommand{\minimizewrt}[1]{\underset{#1}{\minimize}}

\newcommand{\subjectto}{\mbox{subject to}}

\usepackage{xcolor}

\newcommand{\optprobstart}{\begin{equation}}
\newcommand{\optprobend}{\end{equation}}

\newcommand{\optprobstartnn}{\begin{equation*}}
\newcommand{\optprobendnn}{\end{equation*}}

\usepackage{listings}

\usepackage{color}

\newcommand{\mtxstart}{\left[\begin{array}}
\newcommand{\mtxend}{\end{array}\right]}	

\usepackage{tikz}

\newcommand{\tr}{\mathop{\bf tr}}

\newcommand{\vect}{\mathop{\bf vec}}

\usepackage{bm}
\usepackage{wrapfig}
\usepackage{bbm}
\usepackage{algorithm}
\usepackage{algorithmic}
\usepackage{subcaption}
\usepackage{multirow}
\usepackage{xcolor,colortbl,cancel}

\usepackage{pdflscape}
\usepackage{changepage}
\title{A Semismooth Newton Method for Fast, Generic Convex Programming}

\date{}

\author{Alnur Ali\footnotemark[1] \\
Machine Learning Department \\
Carnegie Mellon University \\
\texttt{alnurali@cmu.edu} \and
Eric Wong\footnotemark[1] \\
Machine Learning Department \\
Carnegie Mellon University \\
\texttt{ericwong@cs.cmu.edu} \and
J.~Zico Kolter \\
Computer Science Department \\
Carnegie Mellon University \\
\texttt{zkolter@cs.cmu.edu}}

\begin{document}
\maketitle

\begin{abstract}
We introduce Newton-ADMM, a method for fast
conic optimization.  The basic idea is to view the residuals of consecutive iterates generated by the
alternating direction method of multipliers (ADMM) as a set of fixed point
equations, and then use a \textit{nonsmooth} Newton method to find a solution; we apply the basic idea to the Splitting Cone Solver (SCS), a state-of-the-art method for solving generic conic optimization
problems.  We demonstrate theoretically, by extending the theory of 
\textit{semismooth operators}, that Newton-ADMM converges rapidly (\ie,
quadratically) to a solution; empirically, Newton-ADMM is significantly faster
than SCS on a number of problems.  The method also has essentially no tuning parameters, generates certificates of primal or dual infeasibility, when appropriate, and can be specialized to solve specific convex problems.
\end{abstract}

\setcounter{footnote}{1}
\renewcommand*{\thefootnote}{\fnsymbol{footnote}}
\footnotetext{These authors contributed equally.}

\section{Introduction and related work}
\textit{Conic optimization problems} (or \textit{cone programs}) are convex optimization problems of the form
%(also known as \textit{conic optimization problems})
\begin{equation}
\begin{array}{ll}
\minimizewrt{x \in \mathbf{R}^n} \quad c^T x \quad\quad \subjectto \quad b - Ax \in \mathcal{K},
\end{array}
\label{eq:coneprog}
\end{equation}
where $c \in \reals^{n}, \; A \in \reals^{m \times n}, \; b \in \reals^{m}, \; \mathcal{K}$ are problem data, specified by the user, and $\mathcal{K}$ is a \textit{proper cone} \citep{nesterov1994interior,ben2001lectures,boyd2004convex}; we give a formal treatment of proper cones in Section \ref{sec:bkgd}, but a simple example of a proper cone, for now, is the \textit{nonnegative orthant}, \ie, the set of all points in $\reals^m$ with nonnegative components.  These problems are quite general, encapsulating a number of standard problem classes: \eg, taking $\mathcal{K}$ as the nonnegative orthant yields a linear program; taking $\mathcal{K}$ as the \textit{positive semidefinite cone}, \ie, the space of $m \times m$ positive semidefinite matrices $\symm_{+}^m$, yields a semidefinite program; and taking $\mathcal{K}$ as the \textit{second-order} (or \textit{Lorentz}) \textit{cone} $\{ (x, y) \in \reals^{m-1} \times \reals : \| x \|_2 \leq y \}$ yields a second-order cone program (a quadratic program is a special case).
%can yield a 

Due, in part, to their generality, cone programs have been the focus of much recent work, and additionally form the basis of many convex optimization modeling frameworks, \eg, sdpsol \citep{wu2000sdpsol}, YALMIP \citep{lofberg2005yalmip}, and the CVX family of frameworks \citep{grant2008cvx,diamond2016cvxpy,hong2014cvx}.  These frameworks generally make it easy to quickly solve small and medium-sized convex optimization problems to high accuracy; they work by allowing the user to specify a generic convex optimization problem in a way that resembles its mathematical representation, then convert the problem into a form similar to \eqref{eq:coneprog}, and finally solve the problem.  Primal-dual interior point methods, \eg, SeDuMi \citep{sturm2002implementation}, SDPT3 \citep{toh2012implementation}, and CVXOPT \citep{andersen2011interior}, are common for solving these cone programs.  These methods are useful, as they generally converge to high accuracy in just tens of iterations, but they solve a Newton system on each iteration, and so have difficulty scaling to high-dimensional (\ie, large-$n$) problems.
%tools

In recent work, \citet{o2016conic} use the alternating direction method
of multipliers (ADMM) \citep{boyd2011distributed} to solve generic cone
programs; operator splitting methods (\eg, ADMM, Peaceman-Rachford splitting
\citep{peaceman1955numerical}, Douglas-Rachford splitting
\citep{douglas1956numerical}, and dual decomposition) generally converge to
modest accuracy in just a few iterations, so the approach (called the splitting
conic solver, or SCS) is scalable, and also has a number of other benefits, \eg, provding certificates of primal or dual infeasibility.

In this paper, we introduce a new method (called ``Newton-ADMM'') for solving
large-scale, generic cone programs rapidly to high accuracy.  The basic idea is
to view the usual ADMM recurrence relation as a fixed point iteration, and then
use a truncated, \textit{nonsmooth} Newton method to find a fixed point; to
justify the approach, we extend the theory of \textit{semismooth operators},
coming out of the applied mathematics literature over the last two decades
\citep{mifflin1977semismooth,qi1993nonsmooth,martinez1995inexact,facchinei1996inexact},
although it has received little attention from the machine learning community
\citep{ferris2004semismooth}.  We apply the approach to the fixed
point iteration associated with SCS, to obtain a general purpose conic
optimizer.  We show, under regularity conditions, that Newton-ADMM is quadratically convergent; empirically, Newton-ADMM is significantly faster than SCS, on a number of problems.  Also, Newton-ADMM has essentially no tuning parameters, and generates certificates of infeasibility, helpful in diagnosing problem misspecification.
%when appropriate, 

The rest of the paper is organized as follows.  In Section \ref{sec:bkgd}, we
give the background on cone programs, SCS, and semismooth operators, required to
derive our method for solving generic cone programs, Newton-ADMM.
.  In Section \ref{sec:ours}, we present Newton-ADMM, and establish some of its basic properties.  In Section \ref{sec:theory}, we give various convergence guarantees.  In Section \ref{sec:exps}, we empirically evaluate Newton-ADMM, and describe an extension as a specialized solver.  We conclude with a discussion in Section \ref{sec:disc}.
%also 
% for Newton-ADMM

\section{Background}
\label{sec:bkgd}
We first give some background on cones.  Using this background, we go on to describe SCS, the cone program solver of \citet{o2016conic}, in more detail.  Finally, we give an overview of semismoothness \citep{mifflin1977semismooth}, a generalization of smoothness, central to our Newton method.

\subsection{Cone programming}
We say that a set $\mathcal{C}$ is a \textit{cone} if, for all $x \in \mathcal{C}$, and $\theta \geq 0$, we get that $\theta x \in \mathcal{C}$.  The dual cone $\mathcal{C}^*$, associated with the cone $\mathcal{C}$, is defined as the set $\{ y : y^T x \geq 0, \; \forall x \in \mathcal{C} \}$.  Additionally, a cone $\mathcal{C}$ is a \textit{convex cone} if, for all $x, y \in \mathcal{C}$, and $\theta_1, \theta_2 \geq 0$, we get that $\theta_1 x + \theta_2 y \in \mathcal{C}$.  A cone $\mathcal{C}$ is a \textit{proper cone} if it is (i) convex; (ii) closed; (iii) \textit{solid}, \ie, its interior is nonempty; and (iv) \textit{pointed}, \ie, if both $x, -x \in \mathcal{C}$, then we get that $x = 0$.\vspace{0.1in}

The nonnegative orthant, second-order cone, and positive semidefinite cone are all proper cones \citep[Section 2.4.1]{boyd2004convex}; these cones, along with the \textit{exponential cone} (defined below), can be used to represent most convex optimization problems encountered in practice.  The exponential cone (see, \eg, \citet{serrano2015algorithms}), $\mathcal{K}_{\textrm{exp}}$, is a three-dimensional proper cone, defined as the closure of the epigraph of the perspective of $\exp(x)$, with $x \in \reals$:
\begin{align*}
\mathcal{K}_{\textrm{exp}} 
%& = \cl \left\{ (x,y,\delta) : x \in \reals, \; y > 0, \; z \geq y \exp(x / y) \right\} \\
& = \left\{ (x,y,z) : x \in \reals, \; y > 0, \; z \geq y \exp(x / y) \right\} \cup \left\{ (x, 0, z) : x \leq 0, \; z \geq 0 \right\}.
\end{align*}
Cone programs resembling \eqref{eq:coneprog} were first described by \citet[page 67]{nesterov1994interior}, although special cases were, of course, considered earlier.  Standard references include \citet{ben2001lectures} and \citet[Section 4.6.1]{boyd2004convex}.

\subsection{SCS}
Roughly speaking, SCS is an application of ADMM to a particular feasibility problem arising from the Karush-Kuhn-Tucker (KKT) optimality conditions associated with a cone program.  To see this, consider a reformulation of the cone program \eqref{eq:coneprog}, with slack variable $s \in \mathbf{R}^m$:
\begin{equation}
\begin{array}{ll}
\minimizewrt{x \in \mathbf{R}^n, \, s} \;\; c^T x \quad \subjectto \;\;  Ax + s = b, \; s \in \mathcal{K}.
\end{array}
\label{eq:coneprog2}
\end{equation}
The KKT conditions can be seen, after introducing dual variables $r \in \reals^n, \; y \in \mathcal{K}^*$, for the implicit constraint $x \in \reals^n$ and the explicit constraints, respectively, to be
%associated with
\begin{align*}
A^T y + c = r & \quad \textrm{(stationarity)} \\
Ax + s = b, \; s \in \mathcal{K} & \quad \textrm{(primal feasibility)} \\
r \in \{0\}^n, \; y \in \mathcal{K}^* & \quad \textrm{(dual feasibility)} \\
-c^T x - b^T y = 0 & \quad \textrm{(complementary slackness)},
\end{align*}
where $\mathcal{K}^*$ is the dual cone of $\mathcal{K}$; thus, we can obtain a solution to \eqref{eq:coneprog2}, by solving the KKT system
%Since the cone program \eqref{eq:coneprog2} is convex,
%corresponding
\begin{align}
& \left[
\begin{array}{cc}
0 & A^T \\
-A & 0 \\
-c^T & -b^T
\end{array}
\right]
\left[
\begin{array}{c}
x \\
y
\end{array}
\right]
+
\left[
\begin{array}{c}
c \\
b \\
0
\end{array}
\right]
=
\left[
\begin{array}{c}
r \\
s \\
0
\end{array}
\right], \quad\quad x \in \reals^n, \; y \in \mathcal{K}^*, \; r \in \{0\}^n, \; s \in \mathcal{K}. \label{eq:kktsys}
\end{align}

\paragraph{Self-dual homogeneous embedding.}
When the cone program \eqref{eq:coneprog2} is primal/dual infeasible, there is no solution to the KKT system \eqref{eq:kktsys}; so, consider embedding the system \eqref{eq:kktsys} in a larger system, with new variables $\tau, \kappa$, and solving
\begin{align}
& \left[
\begin{array}{ccc}
0 & A^T & c \\
-A & 0 & b \\
-c^T & -b^T & 0
\end{array}
\right]
\left[
\begin{array}{c}
x \\
y \\
\tau
\end{array}
\right]
=
\left[
\begin{array}{c}
r \\
s \\
\kappa
\end{array}
\right], \quad\quad x \in \reals^n, \; y \in \mathcal{K}^*, \; \tau \in \reals_+, \; r \in \{0\}^n, \; s \in \mathcal{K}, \; \kappa \in \reals_+, \label{eq:kktsysembed}
\end{align}
which is always solvable.  The embedding \eqref{eq:kktsysembed}, due to \citet{ye1994nl}, has a number of other nice properties.  Observe that when $\tau^\star = 1, \, \kappa^\star = 0$ are solutions to the embedding \eqref{eq:kktsysembed}, we recover the KKT system \eqref{eq:kktsys}; it turns out that the solutions $\tau^\star, \kappa^\star$ characterize the primal or dual (in)feasibility of the cone program \eqref{eq:coneprog2}.  In particular, if $\tau^\star > 0, \, \kappa^\star = 0$, then the cone program \eqref{eq:coneprog2} is feasible, with a primal-dual solution $(1/\tau^\star) (x^\star, y^\star, r^\star, s^\star)$; on the other hand, if $\tau^\star = 0, \, \kappa^\star \geq 0$, then \eqref{eq:coneprog2} is primal or dual infeasible (or both), depending on the exact values of $\tau^\star, \, \kappa^\star$ \citep[Section 2.3]{o2016conic}.  The embedding \eqref{eq:kktsysembed} can also be seen as first-order homogeneous, in the sense that $(x^\star, y^\star, \tau^\star, r^\star, s^\star, \kappa^\star)$ being a solution to \eqref{eq:kktsysembed} implies that $k (x^\star, y^\star, \tau^\star, r^\star, s^\star, \kappa^\star)$, for $k \geq 0$, is also a solution.  Finally, viewing the embedding \eqref{eq:kktsysembed} as a feasibility problem, the dual of the feasibility problem turns out to be the original feasibility problem, \ie, the embedding is self-dual.
%benefits, as well
%see, \eg, 

\paragraph{ADMM-based algorithm.}
As mentioned, the embedding \eqref{eq:kktsysembed} can be viewed as the feasibility problem
\begin{equation*}
\begin{array}{ll}
\textrm{find} \quad u,v \quad\quad \subjectto \quad Q u = v, \; (u,v) \in \mathcal{C} \times \mathcal{C}^*,
\end{array}
%\label{eq:feasprob}
\end{equation*}
where we write $\mathcal{C} = \reals^n \times \mathcal{K}^* \times \reals_+, \; \mathcal{C}^* = \{0\}^n \times \mathcal{K} \times \reals_+$,
\begin{align}
Q & = \left[
\begin{array}{ccc}
0 & A^T & c \\
-A & 0 & b \\
-c^T & -b^T & 0
\end{array}
\right], \quad
u = \left[
\begin{array}{c}
x \\
y \\
\tau
\end{array}
\right], \quad
v = \left[
\begin{array}{c}
r \\
s \\
\kappa
\end{array}
\right]. \label{eq:Q}
\end{align}
Introducing new variables $\tilde u, \tilde v \in \reals^k$, where $k = n + m + 1$, and rewriting so that we may apply ADMM, we get:
\begin{equation*}
\begin{array}{ll}
\minimizewrt{u, \, v, \, \tilde u, \, \tilde v} & I_{\mathcal{C} \times \mathcal{C}^*} (u,v) + I_{Q u^\star = v^\star}(\tilde u, \tilde v) \\
\subjectto & \left[ \begin{array}{c} u \\ v \end{array} \right] = \left[ \begin{array}{c} \tilde u \\ \tilde v \end{array} \right],
\end{array}
%\label{eq:admmform}
\end{equation*}
where $I_{\mathcal{C} \times \mathcal{C}^*}$ and $I_{Q u^\star = v^\star}$ are the indicator functions of the product space $\mathcal{C} \times \mathcal{C}^*$, and the affine space of solutions to $Q u = v$, respectively; after simplifying (see \citet[Section 3]{o2016conic}), the ADMM recurrences are just
\begin{align}
\tilde u & \leftarrow (I + Q)^{-1} (u + v). \label{eq:1} \\
u & \leftarrow P_{\mathcal{C}}(\tilde u - v) \label{eq:2} \\
v & \leftarrow v - \tilde u + u, \label{eq:3}
\end{align}
where $P_\mathcal{C}$ denotes the projection onto $\mathcal{C}$.  For the update \eqref{eq:1}, $Q$ is a skew-symmetric matrix, hence $I+Q$ is nonsingular, so the update can be done efficiently via the Schur complement, matrix inversion lemma, and $L D L^T$ factorization.

\paragraph{Projections onto dual cones.}
For the update \eqref{eq:2}, the projection onto $\mathcal{C}$ boils down to separate projections onto the ``free'' cone $\reals^n$, the dual cone of $\mathcal{K}$, and the nonnegative orthant $\reals_+$.  These projections, for many $\mathcal{K}$, are well-known:
%Chapter 8
% (\cf~\citet{boyd2004convex})
%we list them below
\begin{itemize}
\item \textit{Free cone.}  Here, $P_{\mathbf{R}^n}(z) = z$, for $z \in \reals^n$.

\item \textit{Nonnegative orthant, $\mathcal{K}_{\textrm{no}}$.}  The projection onto $\mathcal{K}_{\textrm{no}}$ is simply given by applying the positive part operator:
\begin{equation}
P_{\mathcal{K}_{\textrm{no}}}(z) = \max \{ z,0 \}. \label{eq:projnno}
\end{equation}

\item \textit{Second-order cone, $\mathcal{K}_{\textrm{soc}}$}.  Write $z = (z_1, z_2) \in \reals^m, \; z_1 \in \reals^{m-1}, \; z_2 \in \reals$.  Then the projection is
\begin{equation}
P_{\mathcal{K}_{\textrm{soc}}}(z) = 
\begin{cases}
0, & \| z_1 \|_2 \leq -z_2 \\ %\textrm{if } 
z, & \| z_1 \|_2 \leq z_2 \\
\frac{1}{2} (1 + \frac{z_2}{\| z_1 \|_2}) (z_1, \| z_1 \|_2), & \textrm{otherwise}. %\| z_1 \|_2 \geq |z_2|.
\end{cases} \label{eq:projsoc}
\end{equation}

\item \textit{Positive semidefinite cone, $\mathcal{K}_{\textrm{psd}}$.}  The projection is
\begin{equation}
P_{\mathcal{K}_{\textrm{psd}}}(Z) = \sum_{i} \max \{ \lambda_i, 0 \} q_i q_i^T, \label{eq:projpsd}
\end{equation}
where $Z = \sum_{i} \lambda_i q_i q_i^T$ is the eigenvalue decomposition of $Z$.

\item \textit{Exponential cone, $\mathcal{K}_{\textrm{exp}}$.}  If $z \in \mathcal{K}_{\textrm{exp}}$, then $P_{\mathcal{K}_{\textrm{exp}}}(z) = z$.  If $-z \in \mathcal{K}_{\textrm{exp}}^*$, then $P_{\mathcal{K}_{\textrm{exp}}}(z) = 0$.  If $z_1, z_2 < 0$, \ie, the first two components of $z$ are negative, then $P_{\mathcal{K}_{\textrm{exp}}} = (z_1,\max \{z_2, 0\}, \max \{ z_3, 0 \})$.  Otherwise, the projection is given by
%There are several cases to consider.  
\begin{equation}
\begin{array}{ll}
\underset{\tilde z \in \mathbf{R}^3 : \tilde z_2 > 0}{\argmin} & (1/2) \| \tilde z - z \|_2^2 \\
\subjectto & \tilde z_2 \exp(\tilde z_1 / \tilde z_2) = \tilde z_3,
\end{array}
\label{eq:projexp}
\end{equation}
which can be computed using a Newton method \citep[Section 6.3.4]{parikh2014proximal}.
\end{itemize}

The nonnegative orthant, second-order cone, and positive semidefinite cone are all self-dual, so projecting onto these cones is equivalent to projecting onto their dual cones; to project onto the dual of the exponential cone, we use the Moreau decomposition to get
%\citep[Section 2.5]{parikh2014proximal} 
\begin{equation}
P_{\mathcal{K}_{\textrm{exp}}^*}(z) = z + P_{\mathcal{K}_{\textrm{exp}}}(-z). \label{eq:moreau}
\end{equation} 

\subsection{Semismooth operators}
Here, we give an overview of semismoothness; good references include \citet{ulbrich2011semismooth} and \citet{izmailov2014newton}.  We consider maps $F : \reals^k \to \reals^k$ that are locally Lipschitz, \ie, for all $z_1 \in \reals^k$, and $z_2 \in \mathcal{N}(z_1, \delta)$, where $\mathcal{N}(z_1, \delta)$ is a ball centered at $z_1$ with radius $\delta > 0$, there exists some $L_{z_1} > 0$, such that $\| F(z_1) - F(z_2) \|_2 \leq L_{z_1} \| z_1 - z_2 \|_2$.  By a result known as Rademacher's theorem \citep[Section 3.1.2, Theorem 2]{evans2015measure}, we get that $F$ is differentiable almost everywhere; we let $\mathcal{D}$ denote the points at which $F$ is differentiable, so that $\reals^k \setminus \mathcal{D}$ is a set of measure zero.

\paragraph{The generalized Jacobian.}
\citet{clarke1990optimization} suggested the \textit{generalized Jacobian} as a way to define the derivative of a locally Lipschitz map $F : \reals^k \to \reals^k$, at all points.  The generalized Jacobian is related to the subgradient, as well as the directional derivative, as we discuss later on; the generalized Jacobian, though, turns out to be quite useful for defining effective nonsmooth Newton methods.  The generalized Jacobian $\mathcal{J}(z)$ at a point $z \in \reals^k$ of a map $F : \reals^k \to \reals^k$, is defined as ($\co$ denotes convex hull)
\begin{equation}
\mathcal{J}(z) = \co \left\{ \lim_{i \to \infty} J(z_i) : (z_i) \in \mathcal{D}, \; (z_i) \to z \right\}, \label{eq:genjacob}
\end{equation}
where $J(z_i) \in \reals^{k \times k}$ is the usual Jacobian of $F$ at $z_i$.  Two useful properties of the generalized Jacobian \citep[Proposition 1.2]{clarke1990optimization}: (i) $\mathcal{J}(z)$, at any $z$, is always nonempty; and (ii) if each component $F_i$ is convex, then the $i$th row of any element of $\mathcal{J}(z)$ is just a subgradient of $F_i$ at $z$.

\paragraph{(Strong) semismoothness and consequences.}
We say that a map $F : \reals^k \to \reals^k$ is semismooth if it is locally Lipschitz, and if, for all $z, \delta \in \reals^k$, the limit 
\begin{equation}
\lim_{\delta \to 0, \; J \in \mathcal{J}(z + \delta)} J \delta \label{eq:semismooth}
\end{equation}
exists (see, \eg, \citet[Definition 1]{mifflin1977semismooth} and \citet[Section 2]{qi1993nonsmooth}).  The above definition is somewhat opaque, so various works have provided an alternative characterization of semismoothness: $F$ is semismooth if and only if it is (i) locally Lipschitz; (ii) directionally differentiable, in every direction; and (iii) we get
\begin{equation*}
\lim_{\delta \to 0, \; J \in \mathcal{J}(z + \delta) } \frac{\| F(z + \delta) - F(z) - J \delta \|_2}{\| \delta \|_2} = 0, %\label{eq:semismooth2}
\end{equation*}
\ie, $\| F(z + \delta) - F(z) - J \delta \|_2 = o( \| \delta \|_2 ), \; \delta \to 0$ (see, \eg, \citet[Theorem 2.3]{qi1993nonsmooth}, \citet[Theorem 2.9]{hintermuller2010semismooth}, \citet[page 2]{qi1999survey}, and \citet[Proposition 2]{martinez1995inexact}).  Examples of semismooth functions include $\log(1+|x|)$, all convex functions, and all smooth functions \citep{mifflin1977semismooth,smietanski2007generalized}; on the other hand, $\sqrt{|x|}$ is not semismooth.  A linear combination of semismooth functions is semismooth \citep[Proposition 1.75]{izmailov2014newton}.  Finally, we say that a map is \textit{strongly semismooth} if, under the same conditions as above, we can replace \eqref{eq:semismooth} with
\begin{equation*}
\limsup_{\delta \to 0, \; J \in \mathcal{J}(z + \delta)} \frac{ \| F(z + \delta) - F(z) - J \delta \|_2 }{ \| \delta \|_2^2 } < \infty, %\label{eq:strongsemismooth}
\end{equation*}
\ie, $\| F(z + \delta) - F(z) - J \delta \|_2 = O( \| \delta \|_2^2 ), \; \delta \to 0$ (see \citet[Proposition 2.3]{facchinei1996inexact} and \citet[Definition 1]{facchinei1997nonsmooth}).
%; semismooth functionals are also closed under 

\section{Newton-ADMM and its basic properties}
\label{sec:ours}
Next, we describe Newton-ADMM, our nonsmooth Newton method for generic convex programming; again, the basic idea is to view the ADMM recurrences \eqref{eq:1} -- \eqref{eq:3}, used by SCS, as a fixed point iteration, and then use a nonsmooth Newton method to find a fixed point.  Accordingly, we let
\begin{align*}
F(z) = 
\left[
\begin{array}{c}
\tilde u - (I + Q)^{-1} (u + v) \\
u - P_{\mathcal{C}}(\tilde u - v) \\
%u - \left( P_{\mathbf{R}^n}(\tilde u_x - v_r), P_{\mathcal{K}^*}(\tilde u_y - v_s), P_{\mathbf{R}_+}(\tilde u_\tau - v_\kappa) \right) \\
\tilde u - u
\end{array}
\right],
\end{align*}
which are just the residuals of the consecutive ADMM iterates given by \eqref{eq:1} -- \eqref{eq:3}, and $z = (\tilde u, u, v) \in \reals^{3k}$; multiplying by $\diag(I+Q, I, I)$ to change coordinates gives
\begin{align}
F(z) = 
\left[
\begin{array}{c}
(I + Q) \tilde u - (u + v) \\
u - P_{\mathcal{C}}(\tilde u - v) \\
%u - \left( P_{\mathbf{R}^n}(\tilde u_x - v_r), P_{\mathcal{K}^*}(\tilde u_y - v_s), P_{\mathbf{R}_+}(\tilde u_\tau - v_\kappa) \right) \\
\tilde u - u
\end{array}
\right]. \label{eq:F}
\end{align}

Now, we would like to apply a Newton method to $F$, but projections onto proper cones are not differentiable, in general.  However, for many cones of interest, they are (strongly) semismooth; the following lemma summarizes.
\begin{lemma} \label{lem:projsemi}
Projections onto the nonnegative orthant, second-order cone, and positive semidefinite cone are all strongly semismooth; see, \eg, \citet[Section 1]{kong2009clarke}, \citet[Lemma 2.3]{kanzow2006semismooth}, and \citet[Corollary 4.15]{sun2002semismooth}, respectively.
\end{lemma}

Additionally, we give the following new result, for the exponential cone, which may be of independent interest.
\begin{lemma} \label{lem:projexpsemi}
The projection onto the exponential cone is semismooth.
\end{lemma}

We defer all proofs to the supplement.

Putting the pieces together, the following lemma establishes that $F$, defined in \eqref{eq:F}, is (strongly) semismooth.
\begin{lemma} \label{lem:Fsemi}
When $\mathcal{K}$, from the cone program \eqref{eq:coneprog}, is the nonnegative orthant, second-order cone, or positive semidefinite cone, then the map $F$, defined in \eqref{eq:F}, is strongly semismooth; when $\mathcal{K}$ is the exponential cone, then the map $F$ is semismooth.
\end{lemma}

The preceding results lay the groundwork for us to use a semismooth Newton method \citep{qi1993nonsmooth}, applied to $F$, where we replace the usual Jacobian with any element of the generalized Jacobian \eqref{eq:genjacob}; however, as many have observed \citep{khan2017generalized}, it is not always straightforward to compute an element of the generalized Jacobian.  Fortunately, for us, we can just compute a subgradient of each row of $F$, as the following lemma establishes.
\begin{lemma} \label{lem:subg}
The $i$th row of each element of the generalized Jacobian $\mathcal{J}(z)$ at $z$ of the map $F$ is just a subgradient of $F_i, \; i=1,\ldots,3k$, at $z$.
\end{lemma}

Using the lemma, an element $J \in \reals^{3k \times 3k}$ of the generalized Jacobian of the map $F \in \reals^{3k}$ is then just
\begin{equation}
J = 
\left[
\begin{array}{ccc}
I+ Q & -I & -I \\
 & J_u & \\
I & -I & 0
\end{array}
\right], \label{eq:J}
\end{equation}
where
\begin{equation}
J_u =
\left[
\begin{array}{ccccccccc}
-I & 0 & 0 & I & 0 & 0 & I & 0 & 0 \\
0 & -J_{P_{\mathcal{K}^*}} & 0 & 0 & I & 0 & 0 & J_{P_{\mathcal{K}^*}} & 0 \\
0 & 0 & -\ell & 0 & 0 & 1 & 0 & 0 & \ell
\end{array}
\right] \label{eq:Ju}
\end{equation}
is a $(k \times 3k)$-dimensional matrix forming the second row of $J$; $\ell$ equals 1 if $\tilde u_\tau - v_\kappa \geq 0$ and 0 otherwise; and $J_{P_{\mathcal{K}^*}} \in \reals^{m \times m}$ is the Jacobian of the projection onto the dual cone $\mathcal{K}^*$.  Here and below, we use subscripts to select components, \eg, $\tilde u_\tau$ selects the $\tau$-component of $\tilde u$ from \eqref{eq:Q}, and we write $J$ to mean $J(z)$, where  $z = (\tilde u, u, v) \in \reals^{3k}$.
%next
%, as we discuss later
%the appropriate

\subsection{Final algorithm}
Later, we discuss computing $J_{P_{\mathcal{K}^*}}$, the Jacobian of the projection onto the dual cone $\mathcal{K}^*$, for various cones $\mathcal{K}$; these pieces let us compute an element $J$, given in \eqref{eq:J} -- \eqref{eq:Ju}, of the generalized Jacobian of the map $F$, defined in \eqref{eq:F}, which we use instead of the usual Jacobian, in a semismooth Newton method; below, we describe a way to scale the method to larger problems (\ie, values of $n$).
%make several important remarks, related to the implementation of the method.
%The preceding results 
%, described in \eqref{eq:J}, 
%(\ie, for a differentiable map) 

\paragraph{Truncated, semismooth Newton method.}
The conjugate gradient method is, seemingly, an appropriate choice here, as it only approximately solves the Newton system
\begin{equation}
J \Delta = -F, \label{eq:newtsys}
\end{equation}
with variable $\Delta \in \reals^{3k}$; unfortunately, in our case, $J$ is nonsymmetric, so we appeal instead to the generalized minimum residual method (GMRES) \citep{saad1986gmres}.  We run GMRES until
%(details in the supplement)
\begin{equation}
\| F + J \hat \Delta \|_2 \leq \varepsilon \| F \|_2, \label{eq:trunc}
\end{equation}
where $\hat \Delta$ is the approximate solution from a particular iteration of GMRES, and $\varepsilon$ is a user-defined tolerance; \ie, we run GMRES until the approximation error is acceptable.  After GMRES computes an approximate Newton step, we use backtracking line search to compute a step size.

Now recall, from Section \ref{sec:bkgd}, that $\Delta^\star = 0$ is always a trivial solution to the Newton system \eqref{eq:newtsys}, due to homogeneity; so, we initialize the $\tilde u_\tau, \; u_\tau, \; v_\kappa$-components of $z$ to 1, which avoids converging to the trivial solution.  Finally, we mention that when $\mathcal{K}$, in the cone program \eqref{eq:coneprog}, is the direct product of several proper cones, then $J_u$, in \eqref{eq:Ju}, simply consists of multiple such matrices, just stacked vertically.

We describe the entire method in Algorithm \ref{alg:ours}.  The method has essentially no tuning parameters, since, for all the experiments, we just fix the maximum number of Newton iterations $T = 100$; the backtracking line search parameters $\alpha = 0.001, \; \beta = 0.5$; and the GMRES tolerances $\varepsilon^{(i)} = 1/(i+1)$, for each Newton iteration $i$.  The cost of each Newton iteration is the number of backtracking line search iterations times the sum of two costs: the cost of projecting onto a dual cone and the cost of GMRES, \ie, $O(\max\{n^2,m^2\})$, assuming GMRES returns early.  Similarly, the cost of each ADMM iteration of SCS is the cost of projecting onto a dual cone plus $O(\max\{n^2,m^2\})$.
%for all $i$, where $i$ is a Newton iteration counter

\begin{algorithm}[tb]
\caption{Newton-ADMM for convex optimization}
\label{alg:ours}
\begin{algorithmic}
	\STATE {\bf Input:} problem data $c \in \reals^n, \; \mathcal{A} \in \reals^{m \times n}, \; b \in \reals^m$; cones $\mathcal{K}$; maximum number of Newton iterations $T$; backtracking line search parameters $\alpha \in (0,1/2), \; \beta \in (0,1)$; GMRES approximation tolerances $(\varepsilon^{(i)})_{i=1}^T$
	\STATE {\bf Output:} a solution to \eqref{eq:coneprog2}
	\STATE \textbf{initialize} $\tilde u^{(1)} = u^{(1)} = v^{(1)} = 0$ and $\tilde u^{(1)}_\tau = u^{(1)}_\tau = v^{(1)}_\kappa = 1$ \quad // avoids trivial solution
	\STATE \textbf{initialize} $z^{(1)} = (\tilde u^{(1)}, u^{(1)}, v^{(1)})$
	\FOR{$i=1,\ldots,T$}
	
	\STATE \textbf{compute} $J(z^{(i)}), \; F(z^{(i)})$ \quad // see \eqref{eq:F}, \eqref{eq:J}, Sec.~\ref{sec:jacobs}
	\STATE \textbf{compute} the Newton step $\Delta^{(i)}$, \ie, by approximately solving $J(z^{(i)}) \Delta^{(i)} = - F(z^{(i)})$ using GMRES with approximation tolerance $\varepsilon^{(i)}$ \quad // see \eqref{eq:trunc}
	
	\STATE \textbf{initialize} $t^{(i)} = 1$ \quad // initialize step size $t^{(i)}$
	\WHILE{$\| F(z^{(i)} + t^{(i)} \Delta^{(i)}) \|_2^2 \geq (1 - \alpha t^{(i)}) \| F(z^{(i)}) \|_2^2$}
	\STATE $t^{(i)} = \beta t^{(i)}$ \quad // for backtracking line search
	\ENDWHILE
	
	\STATE \textbf{update} $z^{(i+1)} = z^{(i)} + t^{(i)} \Delta^{(i)}$
	\ENDFOR
	
	\STATE \textbf{return} the $u_x$- divided by the $u_\tau$-components of $z^{(T)}$
\end{algorithmic}
\end{algorithm}

\subsection{Jacobians of projections onto dual cones}
\label{sec:jacobs}
Here, we derive the Jacobians of projections onto the dual cones of the nonnegative orthant, second-order cone, positive semidefinite cone, and the exponential cone; here, we write $J_{P_{\mathcal{K}^*}}$ to mean $J_{P_{\mathcal{K}^*}}(z)$, where  $z = \tilde u_y - v_s \in \reals^m$.
%, required by our Newton method.
%to lighten the notation, 

\paragraph{Nonnegative orthant.}
Since the nonnegative orthant is self-dual, we can simply find a subgradient of each component in \eqref{eq:projnno}, to get that $J_{P_{\mathcal{K}^*}}$ is diagonal with, say, $(J_{P_{\mathcal{K}^*}})_{ii}$ set to 1 if $( \tilde u_y - v_s )_i \geq 0$ and 0 otherwise, for $i=1,\ldots,m$.

\paragraph{Second-order cone.}
Write $z = (z_1, z_2), \; z_1 \in \reals^{m-1}, \; z_2 \in \reals$.  The second-order cone is self-dual, as well, so we can find subgradients of \eqref{eq:projsoc}, to get that
%m$-vector $\tilde u_y - v_s
\begin{equation}
J_{P_{\mathcal{K}^*}} = 
\begin{cases}
0, & \| z_1 \|_2 \leq - z_2 \\
I, & \| z_1 \|_2 \leq z_2 \\
D, & \textrm{otherwise}, %\| z_1 \|_2 \geq |z_2|,
\end{cases}
\label{eq:jacprojsoc}
\end{equation}
where $D$ is a low-rank matrix (details in the supplement).

\paragraph{Positive semidefinite cone.}
The projection map onto the (self-dual) positive semidefinite cone is matrix-valued, so computing the Jacobian is more involved.  We leverage the fact that most implementations of GMRES need only the product $J_{P_{\mathcal{K}^*}} (\vect Z)$, provided by the below lemma using matrix differentials \citep{magnus1995matrix}; here, $\vect$ is the vectorization of a real, symmetric matrix $Z$.
\begin{lemma} \label{lem:jacprojpsd}
Let $Z = Q \Lambda Q^T$ be the eigenvalue decomposition of $Z$, and let $\tilde Z$ be a real, symmetric matrix.  Then
\begin{align*}
& J_{P_{\mathcal{K}_{\textrm{psd}}}}(\vect Z) (\vect \tilde Z) \\
& \quad\quad = \vect \left( (d Q) \max(\Lambda, 0) Q^T + Q (d \max(\Lambda, 0)) Q^T + Q \max(\Lambda, 0) (d Q)^T \right),
\end{align*}
where, here, the $\max$ is interpreted diagonally;
\begin{align*}
d Q_i & = (\Lambda_{ii} I - Z)^+ \tilde Z Q_i; \\
\left[ d \max( \Lambda, 0 ) \right]_{ii} & = I_+(\Lambda_{ii}) Q_i^T \tilde Z Q_i;
\end{align*}
$Z^+$ denotes the pseudo-inverse of $Z$; and $I_+(\cdot)$ is the indicator function of the nonnegative orthant.
\end{lemma}
%matrix-vector 

\paragraph{Exponential cone.}
Recall, from \eqref{eq:projexp}, that the projection onto the exponential cone is not analytic, so computing the Jacobian is much more involved, as well.  The following lemma provides a Newton method for computing the Jacobian, using the KKT conditions for \eqref{eq:projexp} and differentials.
\begin{lemma} \label{lem:jacprojexp}
Let $z \in \reals^3$.  Then $J_{P_{\mathcal{K}_{\textrm{exp}}^*}}(z) = I - J_{P_{\mathcal{K}_{\textrm{exp}}}}(-z)$, where
\begin{align*}
J_{P_{\mathcal{K}_{\textrm{exp}}}}(z) =
\begin{cases}
I, & z \in \mathcal{K}_{\textrm{exp}} \\
-I, & z \in \mathcal{K}_{\textrm{exp}}^* \\
\diag(1, I_+(z_2), I_+(z_3)), & z_1, z_2 < 0;
\end{cases}
\end{align*}
otherwise, $J_{P_{\mathcal{K}_{\textrm{exp}}}}(z)$ is a particular 3x3 matrix given in the supplement, due to space constraints.
\end{lemma}

\section{Convergence guarantees}
\label{sec:theory}
Here, we give some convergence results for Newton-ADMM, the method presented in Algorithm \ref{alg:ours}.

First, we show that, under standard regularity assumptions, the iterates $(z^{(i)})_{i=1}^\infty$ generated by Algorithm \ref{alg:ours} are \textit{globally convergent}, \ie, given some initial point, the iterates converge to a solution of $F(z) = 0$, where $i$ is a Newton iteration counter.  We break the statement (and proof) of the result up into two cases.  Theorem \ref{thm:global1} establishes the result, when the sequence of step sizes $(t^{(i)})_{i=1}^\infty$ converges to some number bounded away from zero and one.  Theorem \ref{thm:global2} establishes the result when the step sizes converge zero.

%We start by stating our regularity assumptions.
Below, we state our regularity conditions, which are similar to those given in \citet{han1992globally, martinez1995inexact, facchinei1996inexact}; we elaborate in the supplement.

\begin{itemize}
\item[\textbf{A1.}] For Theorem \ref{thm:global1}, we assume $\limsup_{i \to \infty} t^{(i)} < 1$.

\item[\textbf{A2.}] For Theorem \ref{thm:global2}, we assume $\limsup_{i \to \infty} t^{(i)} = 0$.

\item[\textbf{A3.}] For Theorem \ref{thm:global2}, we assume (i) that the GMRES approximation tolerances $\varepsilon^{(i)}$ are uniformly bounded by $\varepsilon$ as in $\varepsilon^{(i)} \leq \varepsilon < 1 - \alpha^{1/2}$, (ii) that $(\varepsilon^{(i)})_{i=1}^\infty \to 0$, and (iii) that $\varepsilon^{(i)} = O( \| F(z^{(i)}) \|_2 )$.

\item[\textbf{A4.}] For Theorem \ref{thm:global2}, we assume, for every convergent sequence $(z^{(i)})_{i=1}^\infty \to z$, $(\gamma^{(i)})_{i=1}^\infty$ satisfying assumption (A2) above, and $(\Delta^{(j)})_{j=1}^\infty \to \Delta$, that
\begin{align*}
& \lim_{\substack{i,j \to \infty}} \frac{\| F(z^{(i)} + \gamma^{(i)} \Delta^{(j)}) \|_2^2 - \| F(z^{(i)}) \|_2^2}{\gamma^{(i)}} \leq \lim_{\substack{i,j \to \infty}} \alpha^{1/2} F(z^{(i)})^T \hat F(z^{(i)}, \Delta^{(j)}),
\end{align*}
%\\ & \quad\quad
where, for notational convenience, we write
\[
\hat F(z^{(i)}, \Delta^{(j)}) = J(z^{(i)}) \Delta^{(j)}.
\]
%(note that there is intentionally no minus sign on the right-hand side here).

\item[\textbf{A5.}] For Theorem \ref{thm:global2}, we assume, for all $z \in \reals^{3k}$ and $\Delta \in \reals^{3k}$, and for some $C_2 > 0$, that
\[
C_2 \| \Delta \|_2 \leq \| \hat F(z, \Delta) \|_2.
\]

\item[\textbf{A6.}] For Theorem \ref{thm:local}, we assume, for all $z \in \reals^{3k}, \; J(z) \in \mathcal{J}(z)$, (i) that $\| J(z) \|_2 \leq C_3$, for some constant $C_3 > 0$; and (ii) that every element of $\mathcal{J}(z)$ is invertible.
\end{itemize}
%\end{small}

The two global convergence results are given below; the proofs are based on arguments in \citet[Theorem 5a]{martinez1995inexact}, but we use fewer user-defined parameters, and a different line search method.
%Due to space constraints, our regularity conditions are in the supplement.

\begin{theorem}[Global convergence, with $\limsup_{i \to \infty} t^{(i)} = t$, for some $0 < t < 1$]
\label{thm:global1}
Assume condition (A1) stated above.  Then $\lim_{i \to \infty} F(z^{(i)}) = 0$.
\end{theorem}
% in the supplement

\begin{theorem}[Global convergence, with $\limsup_{i \to \infty} t^{(i)} = 0$]
\label{thm:global2}
Assume conditions (A2), (A3), (A4), and (A5) stated above.  Suppose the sequence $(z^{(i)})_{i=1}^\infty$ converges to some $z \in \reals^{3k}$.  Then $F(z) = 0$.
\end{theorem}
% in the supplement

Next, we show, in Theorem \ref{thm:local}, that when $F$ is strongly semismooth, \ie, $\mathcal{K}$ is the nonnegative orthant, second-order cone, or positive semidefinite cone, the iterates $(z^{(i)})_{i=1}^\infty$ generated by Algorithm \ref{alg:ours} are locally quadratically convergent; the proof is similar to that of \citet[Theorem 3.2b]{facchinei1996inexact}, for semismooth maps.
%are \textit{locally quadratically convergent}, \ie, we get, for large enough $i$ and some $C > 0$, that
%\[
%\lim_{i \to \infty} \frac{ |z^{(i+1)} - z| }{ (z^{(i)} - z)^2 } = C.
%\]
%The proof uses a result from \citet{facchinei1996inexact}, where a similar claim is made, but no proof is given.
%a similar claim is made, but no proof is given.

\begin{theorem}[Local quadratic convergence] \label{thm:local}
Assume condition (A6) stated above.  Then the sequence of iterates $(z^{(i)})_{i=1}^\infty \to z$ generated by Algorithm \ref{alg:ours} converges quadratically, with $F(z) = 0$, for large enough $i$.
\end{theorem}

When $\mathcal{K}$ is the exponential cone, \ie, $F$ is semismooth, the iterates generated by Algorithm \ref{alg:ours} are locally superlinearly convergent \citep[Theorem 3.2b]{facchinei1996inexact}.
%$(z^{(i)})_{i=1}^T$ 
%are \textit{locally superlinearly convergent} \citep[Theorem 3b]{facchinei1996inexact}, \ie, we get, for large enough $i$ and some $C > 0$, that
%\[
%\lim_{i \to \infty} \frac{ |z^{(i+1)} - z| }{ |z^{(i)} - z| } = 0.
%\]

\section{Numerical examples}
\label{sec:exps}
Next, we present an empirical evaluation of Newton-ADMM, on several problems; in these, we directly compare to SCS, which Newton-ADMM builds on, as it is the most relevant benchmark for us (\citet{o2016conic} observe that, with an optimized implementation, SCS outperforms SeDuMi, as well as SDPT3).  We evaluate, for both methods, the time taken to reach the solution as well as the optimal objective value; we obtained these by running an interior point method \citep{andersen2011interior} to high accuracy.  Table \ref{tab:sizes} describes the problem sizes, for both the cone form of \eqref{eq:coneprog}, as well as the familiar form that the problem is usually written in. Later, we also describe extending Newton-ADMM to accelerate any ADMM-based algorithm, applied to \textit{any} convex problem; here, we compare to state-of-the-art baselines for specific problems.
%several
%$n = 
%For all the experiments, we set the problem sizes to 1,000, except for the portfolio optimization example, where we set it to 10,000.  

\begin{table}%[t]
\caption{Problem sizes, for the cone form ($n,m$) of \eqref{eq:coneprog}, and the familiar form ($p,N$) that the problem is usually written in.}
\label{tab:sizes}
%\vskip 0.15in
\begin{center}
%\begin{scriptsize}
\begin{sc}
\begin{tabular}{lccccc}
\hline
%\abovespace\belowspace
Problem & $n$ & $m$ & $p$ & $N$ & Cones \\
\hline
%\abovespace
Linear prog. & 600 & 1,200 & 600 & 300 & \textnormal{$\mathcal{K}_{\textrm{no}}$} \\
Portfolio opt. & 2,501 & 2,504 & 2,500 & -- & \textnormal{$\mathcal{K}_{\textrm{soc}}, \mathcal{K}_{\textrm{no}}$} \\
Logistic reg. & 3,200 & 7,200 & 100 & 1,000 & \textnormal{$\mathcal{K}_{\textrm{exp}}, \mathcal{K}_{\textrm{no}}$} \\
%\belowspace
Robust PCA & 4,376 & 8,103 & 25 & 25 & \textnormal{$\mathcal{K}_{\textrm{psd}}, \mathcal{K}_{\textrm{no}}$} \\
\hline
\end{tabular}
\end{sc}
%\end{scriptsize}
\end{center}
\vskip -0.1in
\end{table}
%, \eg, \eqref{eq:portopt},

\subsection{Random linear programs (LPs)}
We compare Newton-ADMM and SCS on a linear program
\begin{equation*}
\begin{array}{ll}
\minimizewrt{x \in \mathbf{R}^p} \quad c^T x \quad\quad \subjectto \quad Gx = h, \; x \geq 0,
\end{array}
\end{equation*}
where $c \in \reals^p, \; G \in \reals^{N \times p}, \; h \in \reals^N$ are problem data, and the inequality is interpreted elementwise.  To ensure primal feasibility, we generated a solution $x^\star$ by sampling its entries from a normal distribution, then projecting onto the nonnegative orthant; we generated $G$ (with $p=600, \; N=300$, so $G$ is wide) by sampling entries from a normal distribution, then taking $h = G x^\star$.  To ensure dual feasibility, we generated dual solutions $\nu^\star, \; \lambda^\star$, associated with the equality and inequality constraints, by sampling their entries from a normal and $\textrm{Uniform}(0,1)$ distribution, respectively; to ensure complementary slackness, we set $c = -G^T \nu^\star + \lambda^\star$.  Finally, to put the linear program into the cone form of \eqref{eq:coneprog}, and hence \eqref{eq:coneprog2}, we just take
\[
A = 
\left[
\begin{array}{c}
G \\
-G \\
I
\end{array}
\right], \quad
b = 
\left[
\begin{array}{c}
h \\
-h \\
0
\end{array}
\right], \quad
\mathcal{K} = \mathcal{K}_{\textrm{no}}.
\]

The first column of Figure \ref{fig:panel} presents the time taken, by both Newton-ADMM and SCS, to reach the optimal objective value, as well as to reach the solution; we see that Newton-ADMM outperforms SCS in both metrics.

\subsection{Minimum variance portfolio optimization}
We consider a \textit{minimum variance portfolio optimization} problem (see, \eg, \citet{khare2015convex,ali2016generalized}),
\begin{equation}
%\begin{array}{ll}
\minimizewrt{\theta \in \mathbf{R}^p} \quad \theta^T \Sigma \theta \quad\quad \subjectto \quad \ones^T \theta = 1,
%\end{array}
\label{eq:portopt}
\end{equation}
where, here, the problem data $\Sigma \in \symm_{++}^p$ is the covariance matrix associated with the prices of $p=2,500$ assets; we generated $\Sigma$ by sampling a positive definite matrix.  The goal of the problem is to allocate wealth across $p$ assets such that the overall risk is minimized; shorting is allowed.  Putting the above problem into the cone form of \eqref{eq:coneprog} yields, for $\mathcal{K}$, the direct product of the second-order cone and the nonnegative orthant (details in the supplement).  The second column of Figure \ref{fig:panel} shows the results; we again see that Newton-ADMM outperforms SCS.
%of a $(p+1)$-dimensional second-order cone, and a two-dimensional nonnegative orthant; the overall problem size ends up being $n=p+1$ (details in the supplement).

\subsection{$\ell_1$-penalized logistic regression}
We consider $\ell_1$-penalized logistic regression, \ie,
%\vskip -0.2in
\begin{equation}
\begin{array}{ll}
\minimizewrt{\theta \in \mathbf{R}^p} \;\, \sum_{i=1}^N \log( 1 + \exp(y_i X_{i \cdot} \theta)) + \lambda \| \theta \|_1,
\end{array}
\label{eq:logreg}
\end{equation}
%\vskip -0.1in
where, here, $y \in \reals^N$ here is a response vector; $X \in \reals^{N \times p}$ is a data matrix, with $X_{i \cdot}$ denoting the $i$th row of $X$; and $\lambda \geq 0$ is a tuning parameter.  We generated $p = 100$ sparse underlying coefficients $\theta^\star$, by sampling entries from a normal distribution, then setting $\approx90\%$ of the entries to zero; we generated $X$ (with $N = 1,000$) by sampling its entries from a normal distribution, then set $y = X \theta^\star + \delta$, where $\delta$ is (additive) Gaussian noise.  For simplicity, we set the tuning parameter $\lambda = 1$.  Putting the above problem into the cone form of \eqref{eq:coneprog} yields, for $\mathcal{K}$, the direct product of the exponential cone and the nonnegative orthant (details in the supplement); the problem size in cone form ends up being large (see Table \ref{tab:sizes}).  In the third column of Figure \ref{fig:panel}, we see that Newton-ADMM outperforms SCS.
%$2N$ exponential cones, and a $(N + 2p)$-dimensional nonnegative

\subsection{Robust principal components analysis (PCA)}
Finally, we consider robust PCA,
%principal components analysis
%\vskip -0.2in
\begin{equation}
%\begin{array}{ll}
\minimizewrt{L, S \in \mathbf{R}^{N \times p}} \; \| L \|_* \;\; \subjectto \; \| S \|_1 \leq \lambda, \; L + S = X,
%\end{array}
\label{eq:robust}
\end{equation}
%\vskip -0.24in
where $\| \cdot \|_*$ and $\| \cdot \|_1$ are the nuclear and elementwise $\ell_1$-norms, respectively, and $X \in \reals^{N \times p}, \; \lambda \geq 0$ \citep[Equation 1.1]{candes2011robust}.  We generated a low-rank matrix $L^\star$, with rank $\approx \frac{1}{2}N$; a sparse matrix $S^\star$, by sampling entries from $\textrm{Uniform}(0,1)$, then setting $\approx90\%$ of the entries to zero; and finally set $X = L^\star + S^\star$.  We set $\lambda = 1$.  The goal is to decompose the obsevations $X$ into low-rank $L$ and sparse $S$ components.  Putting the above problem into the cone form of \eqref{eq:coneprog} yields, for $\mathcal{K}$, the direct product of the positive semidefinite cone and nonnegative orthant (details in the supplement).  We see that Newton-ADMM and SCS are comparable, in the fourth column of Figure \ref{fig:panel}.
%underlying
%Using the same $N,p$ as above, w
%; the problem size in cone form ends up being large (see Table \ref{tab:sizes})

\begin{figure*}[ht]
\vskip 0.2in
\begin{center}
\includegraphics[width=0.23\textwidth]{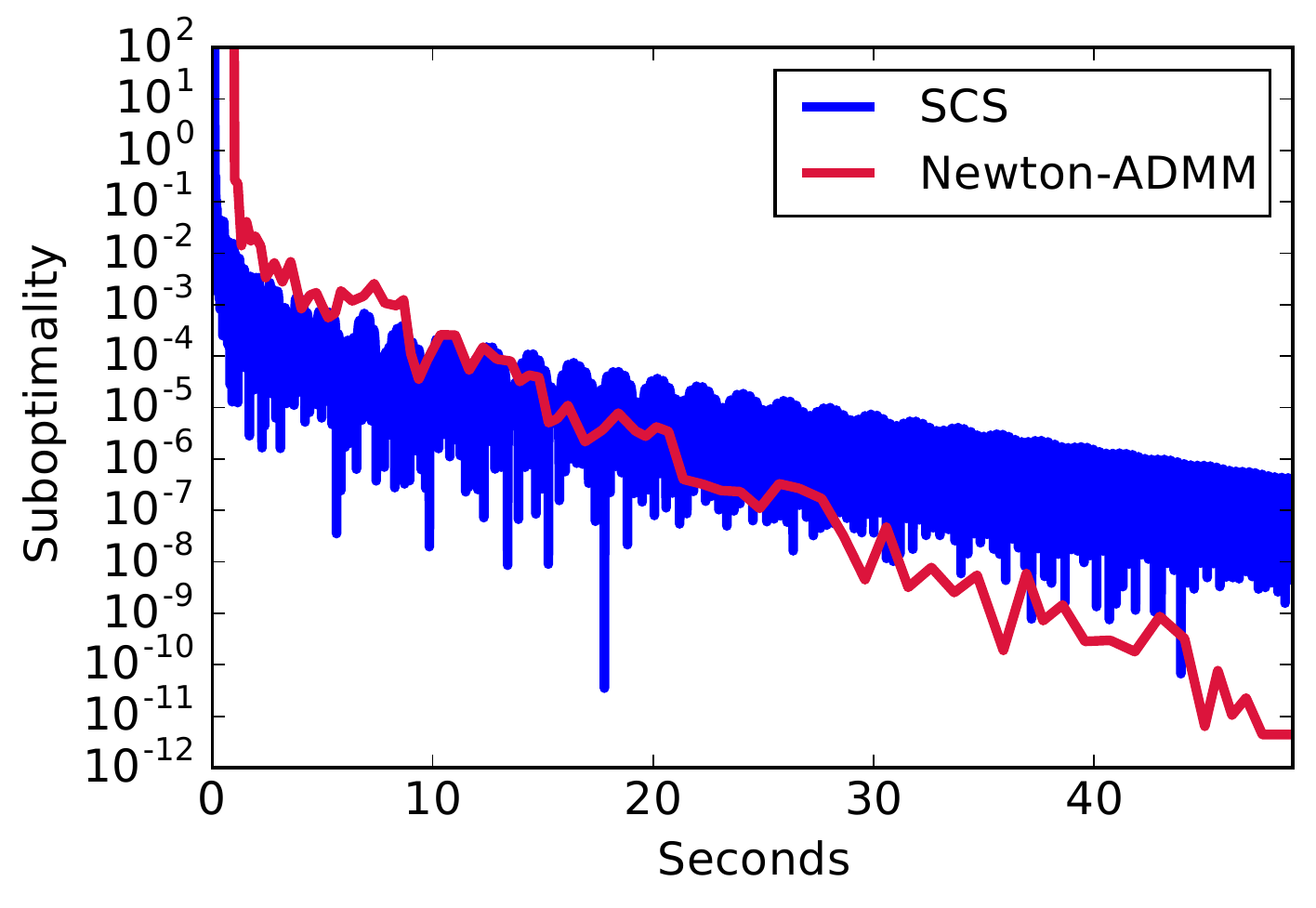}
\hfill
\includegraphics[width=0.23\textwidth]{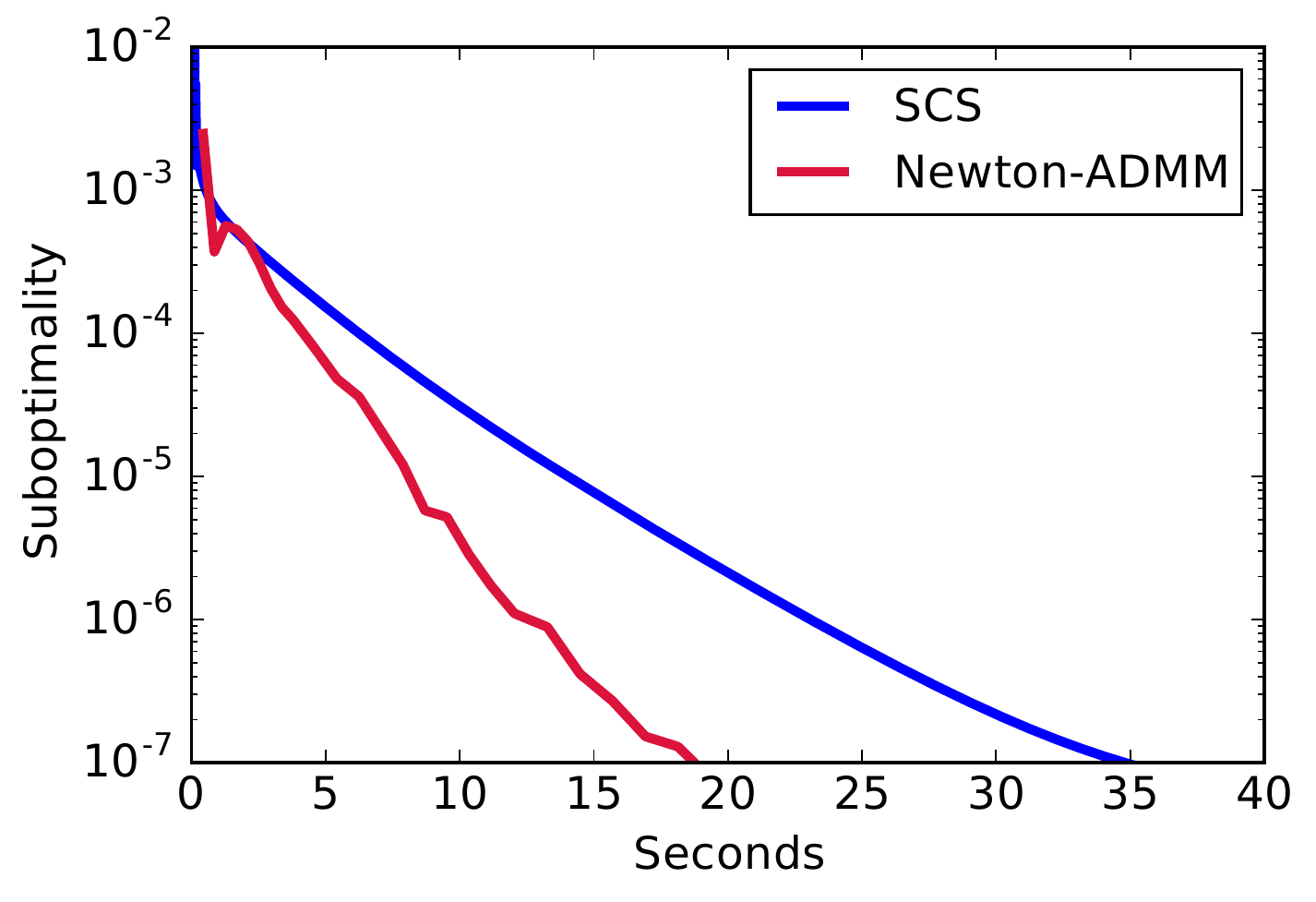}
\hfill
\includegraphics[width=0.23\textwidth]{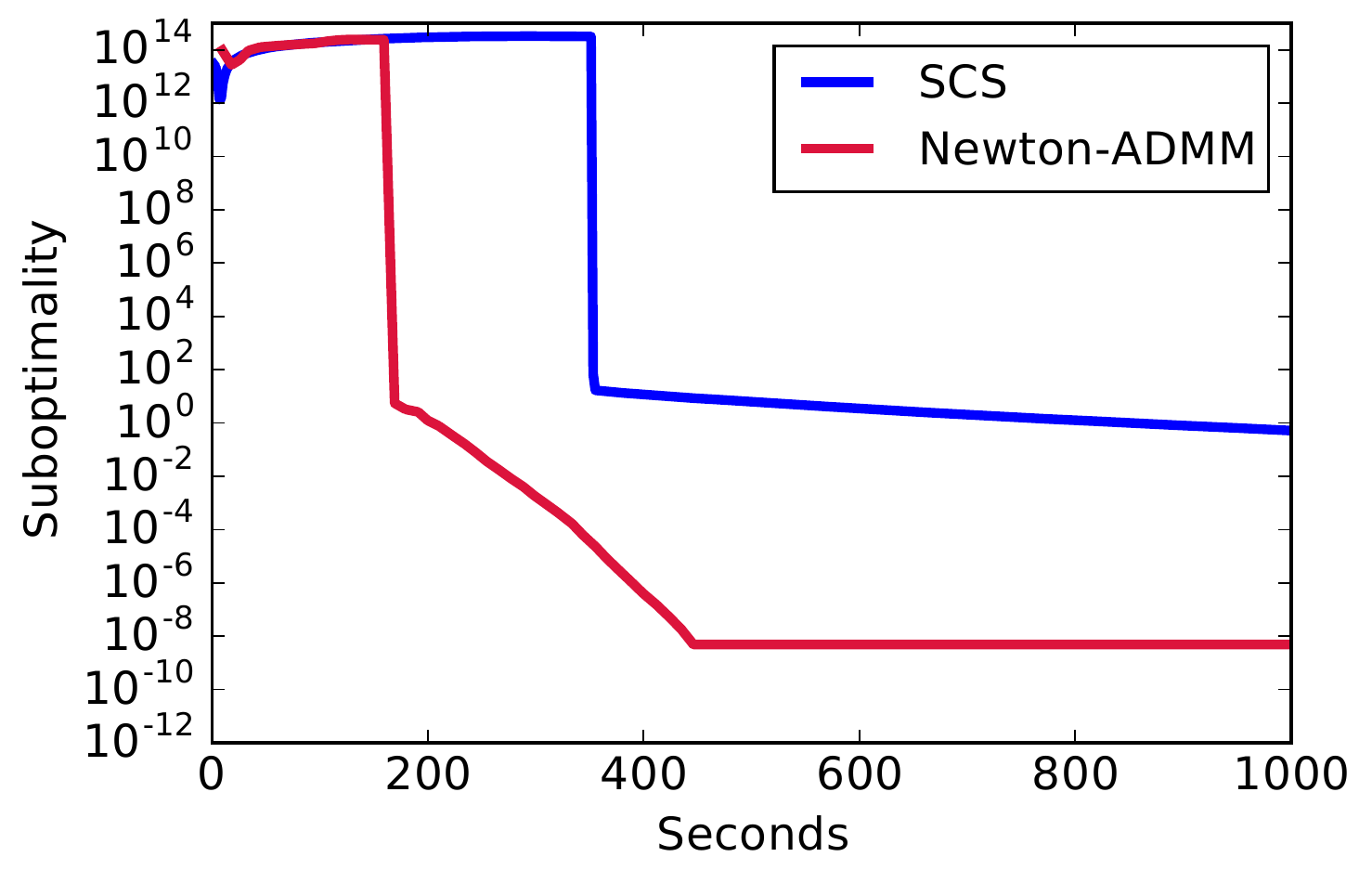}
\hfill
\includegraphics[width=0.23\textwidth]{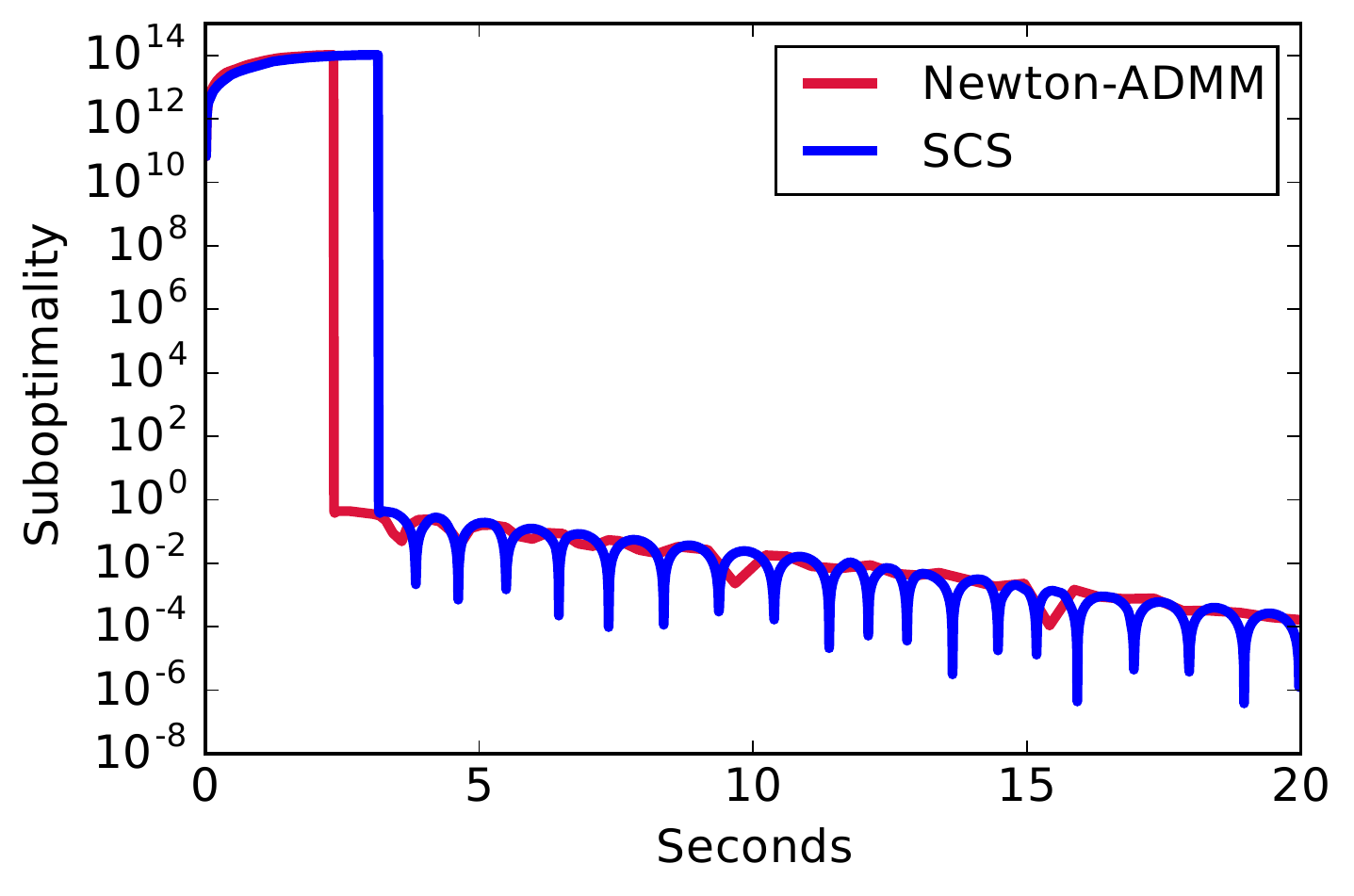}
\\
\includegraphics[width=0.23\textwidth]{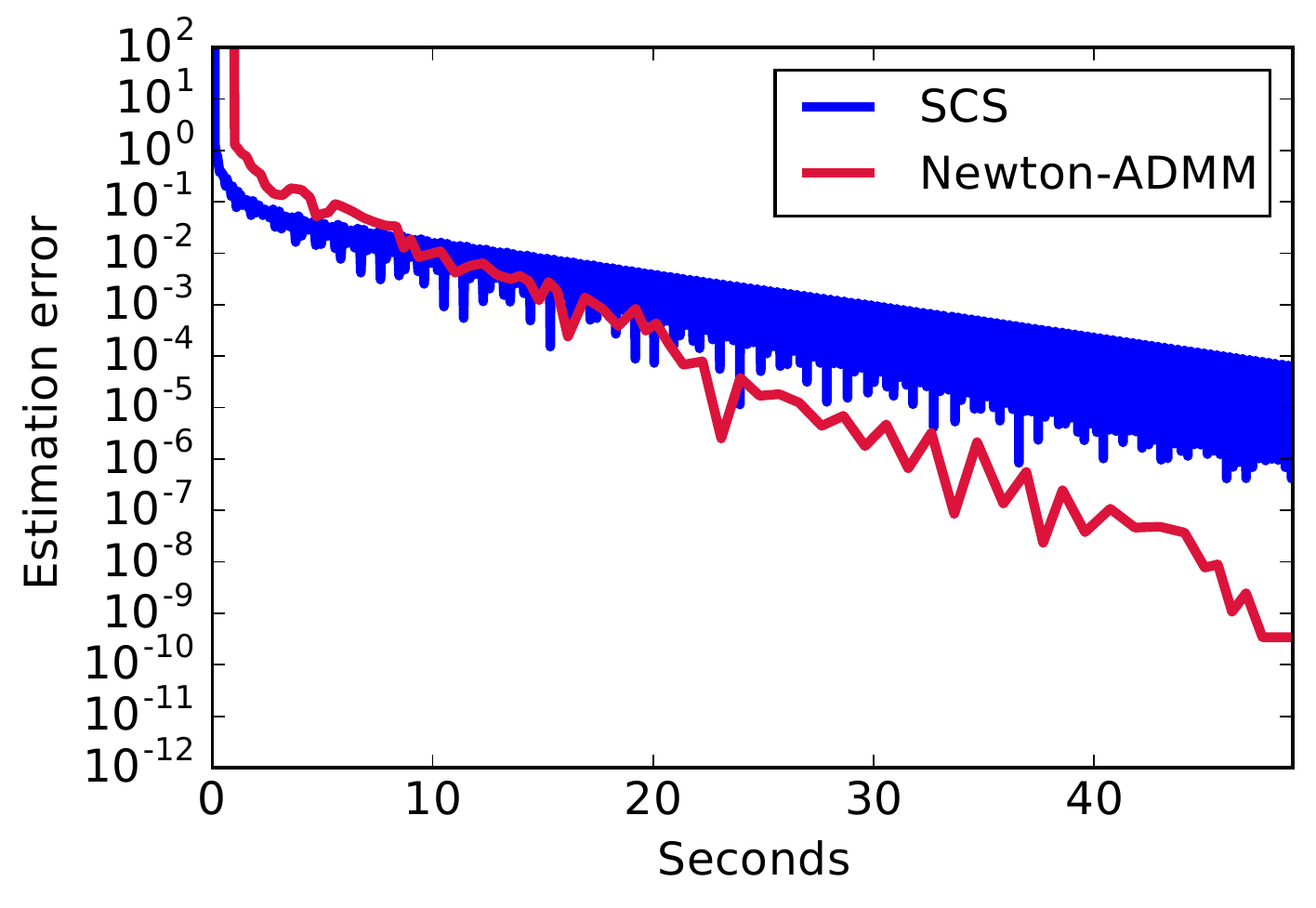}
\hfill
\includegraphics[width=0.23\textwidth]{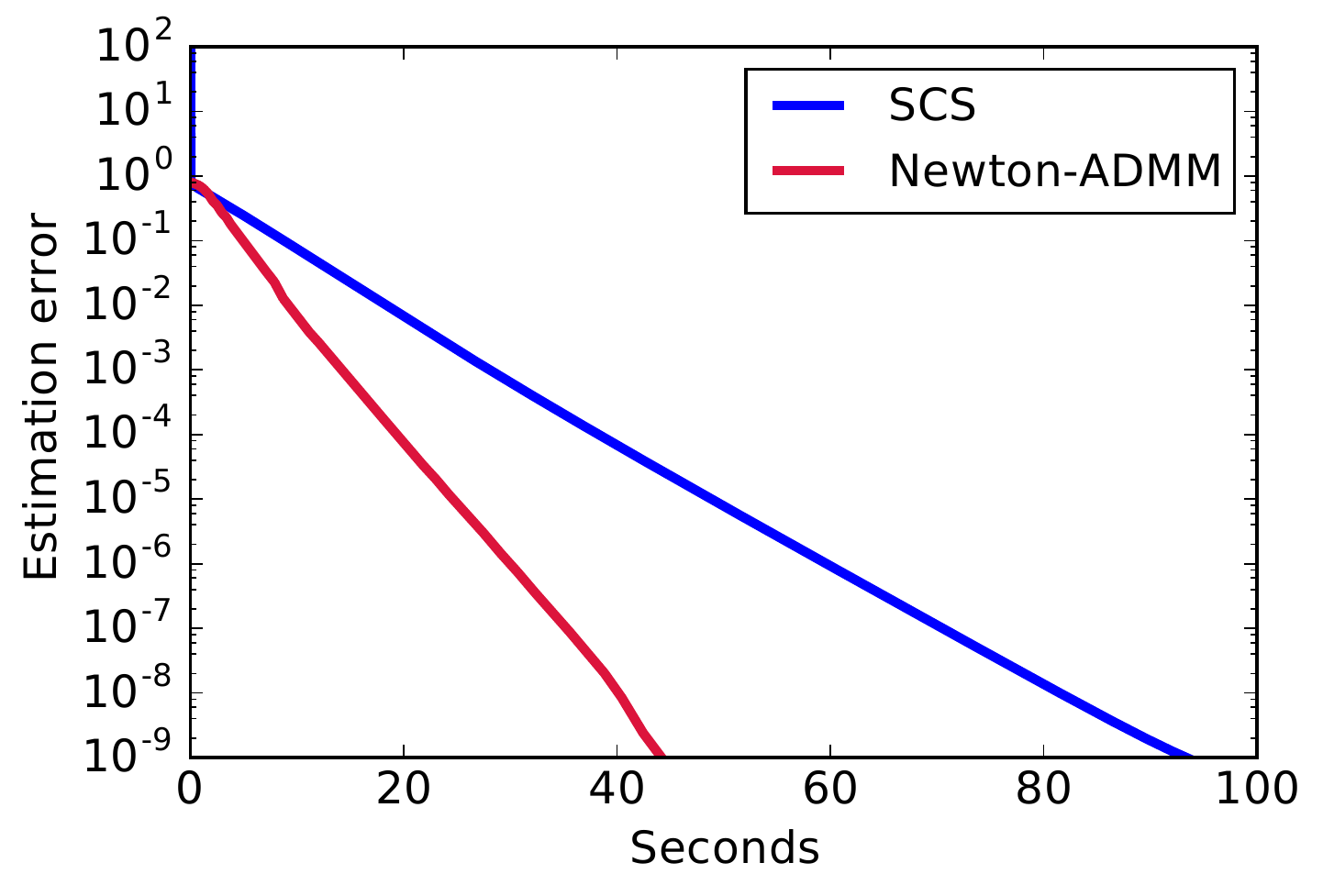}
\hfill
\includegraphics[width=0.23\textwidth]{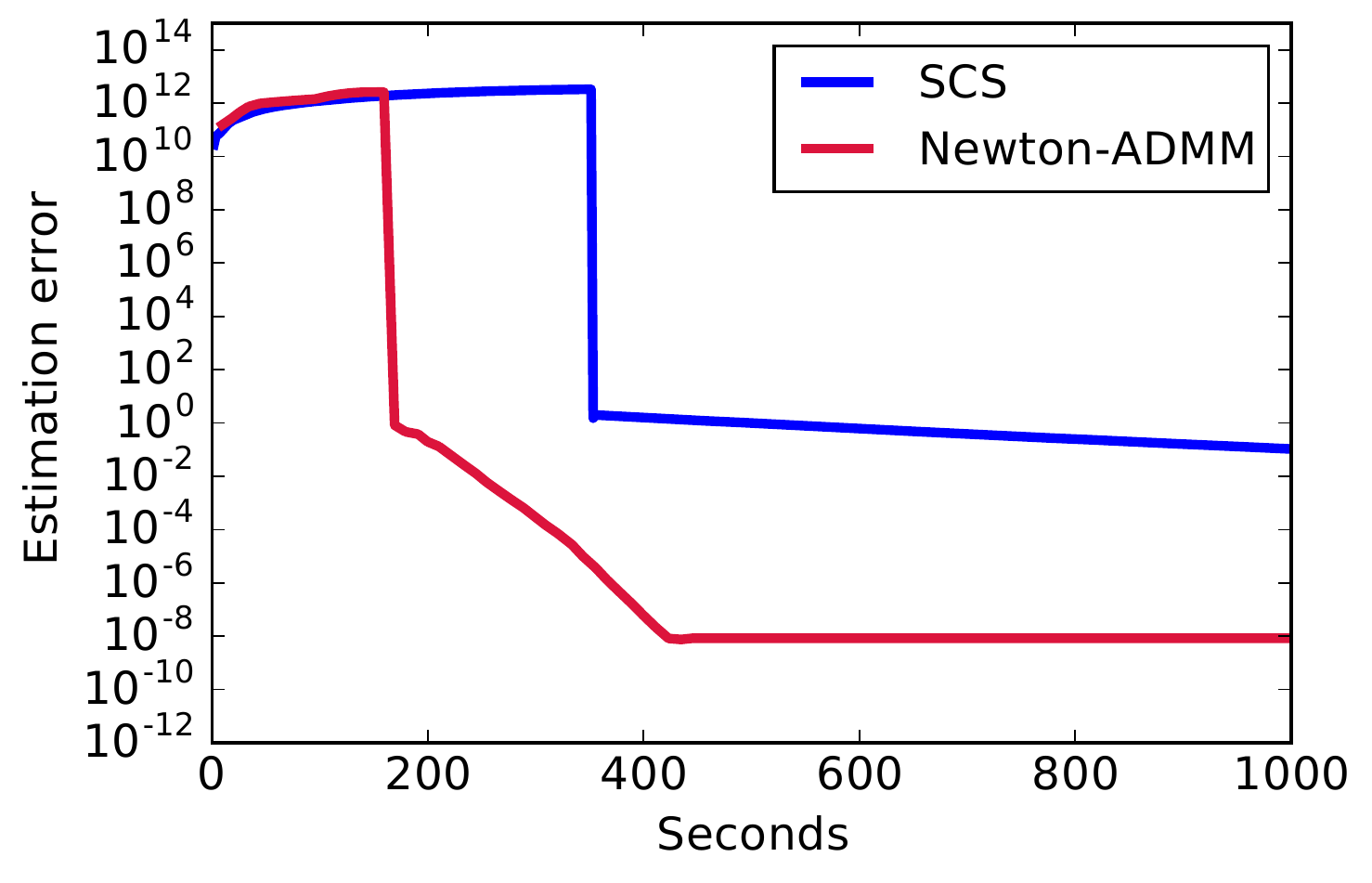}
\hfill
\includegraphics[width=0.23\textwidth]{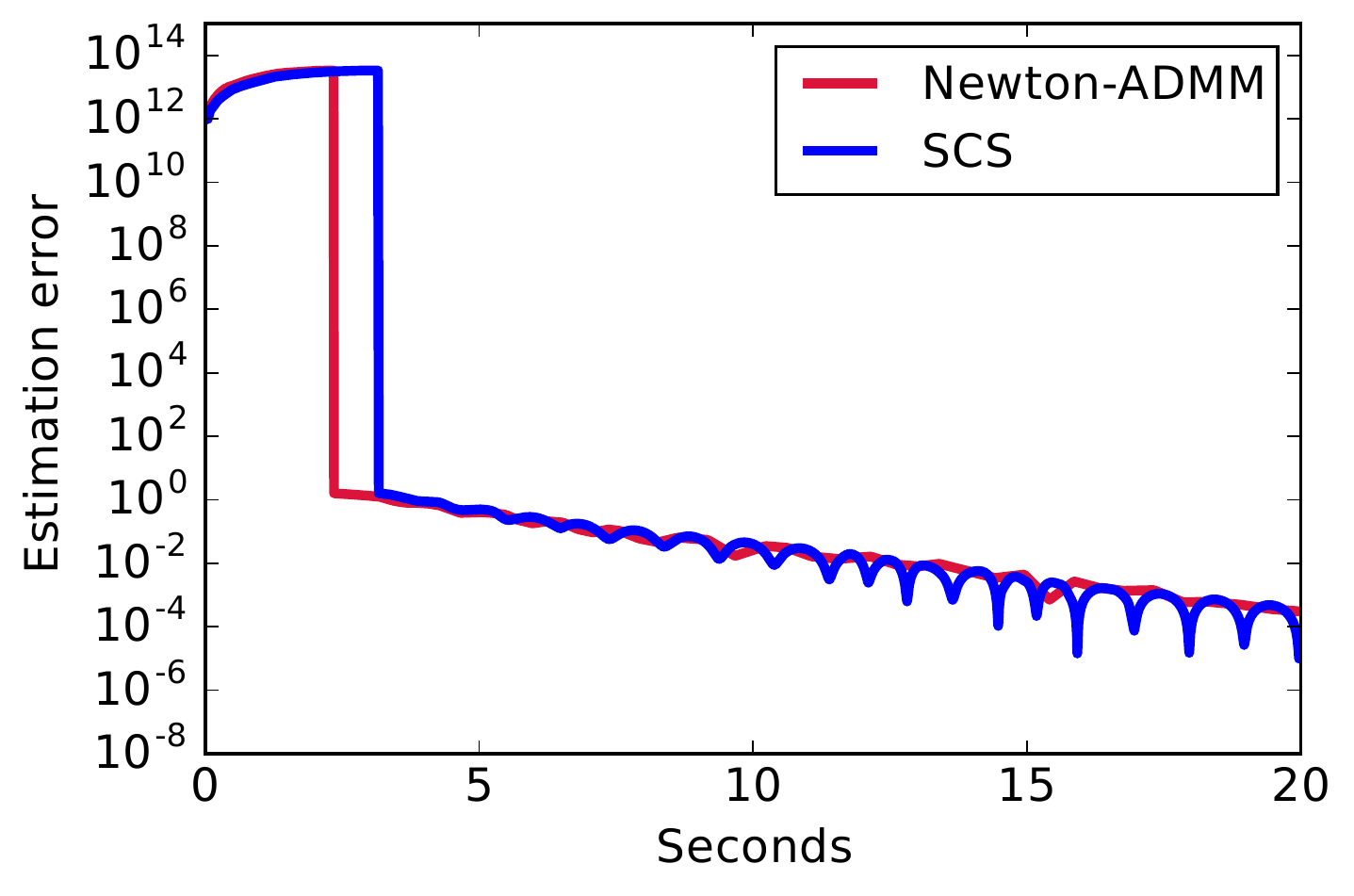}
\caption{Comparison of Newton-ADMM and SCS \citep{o2016conic}, on several convex problems.  Columns, from left to right: linear programming, portfolio optimization, $\ell_1$-penalized logistic regression, robust PCA.  Top row: wallclock time vs.~log-distance to the optimal objective value, obtained by running an interior point method.  Bottom row: wallclock time vs.~log-distance, in a Euclidean norm sense, to the solution.  Each plot is one representative run out of 20 (the variance was negligible).  Best viewed in color.}
\label{fig:panel}
\end{center}
\vskip -0.2in
\end{figure*} 
%to high accuracy
%optimization
%principal components analysis

\subsection{Extension as a specialized solver}
Finally, we observe that the basic idea of treating the residuals of consecutive ADMM iterates as a fixed point iteration, and then finding a fixed point using a Newton method, is completely general, \ie, the same idea can be used to accelerate (virtually) any ADMM-based algorithm, for a convex problem.  To illustrate, consider the lasso problem,
%\vskip -0.15in
\begin{equation}
\begin{array}{ll}
\minimizewrt{\theta \in \mathbf{R}^p} & (1/2) \| y - X \theta \|_2^2 + \lambda \| \theta \|_1,
\end{array}
\label{eq:lasso}
\end{equation}
%\vskip -0.1in
where $y \in \reals^N, \; X \in \reals^{N \times p}, \; \lambda \geq 0$; the ADMM recurrences \citep[Section 6.4]{parikh2014proximal} are
\begin{align}
\theta & \leftarrow (X^TX + \rho I)^{-1}(X^T y + \rho(\kappa - \mu)) \label{eq:4} \\
\kappa & \leftarrow S_{\lambda / \rho} (\theta + \mu) \label{eq:5} \\
\mu & \leftarrow \mu + \theta - \kappa, \label{eq:6}
\end{align}
where $\rho > 0, \; \kappa, \mu \in \reals^p$ are the tuning parameter and auxiliary variables, introduced by ADMM, respectively, and $S_{\lambda / \rho}(\cdot)$ is the soft-thresholding operator.  The map $F : \reals^{3p} \to \reals^{3p}$, from \eqref{eq:F}, with components set to the residuals of the ADMM iterates given in \eqref{eq:4} -- \eqref{eq:6}, is then
\begin{align*}
F(z) = 
\left[
\begin{array}{c}
(X^TX + \rho I) \theta - (X^T y + \rho(\kappa - \mu)) \\
\kappa - S_{\lambda / \rho} (\theta + \mu) \\
\kappa - \theta
\end{array}
\right],
\end{align*}
where $z = (\theta, \kappa, \mu) \in \reals^{3p}$, and we also changed coordinates, similar to before.  An element $J \in \reals^{3p \times 3p}$ of the generalized Jacobian of $F$ is then
\begin{equation*}
J = 
\left[
\begin{array}{ccc}
X^T X + \rho I & -\rho I & \rho I \\
-D & I & D \\
-I & I & 0
\end{array}
\right],
\end{equation*}
where $D \in \reals^{p \times p}$ is diagonal with $D_{ii}$ set to 1 if $|\theta_i + \mu_i| > \lambda / \rho$ and 0 otherwise, for $i=1,\ldots,m$.

In the left panel of Figure \ref{fig:lasso}, we compare a specialized Newton-ADMM applied \textit{directly} to the lasso problem \eqref{eq:lasso}, with the ADMM algorithm for \eqref{eq:4} -- \eqref{eq:6}, a proximal gradient method \citep{beck2009fast}, and a heavily-optimized implementation of coordinate descent \citep{friedman2007pathwise}; we set $p = 400, \; N = 200, \; \lambda = 10, \; \rho = 1$.  Here, the specialized Newton-ADMM is quite competitive with these strong baselines; the specialized Newton-ADMM outperforms Newton-ADMM applied to the cone program \eqref{eq:coneprog2}, so we omit the latter from the comparison.  \citet{stella2016forward} recently described a related approach.
%the general 

In the right panel of Figure \ref{fig:lasso}, we present a similar comparison, for sparse inverse covariance estimation, with the QUIC method of \citet{hsieh2014quic}; Newton-ADMM clearly performs best ($p = N = 1,000, \; \lambda = \rho = 1$, details in the supplement).  
%the
%\ref{fig:spinvcov}

\begin{figure}[ht]
%\vskip 0.2in
\begin{center}
\includegraphics[width=0.49\columnwidth]{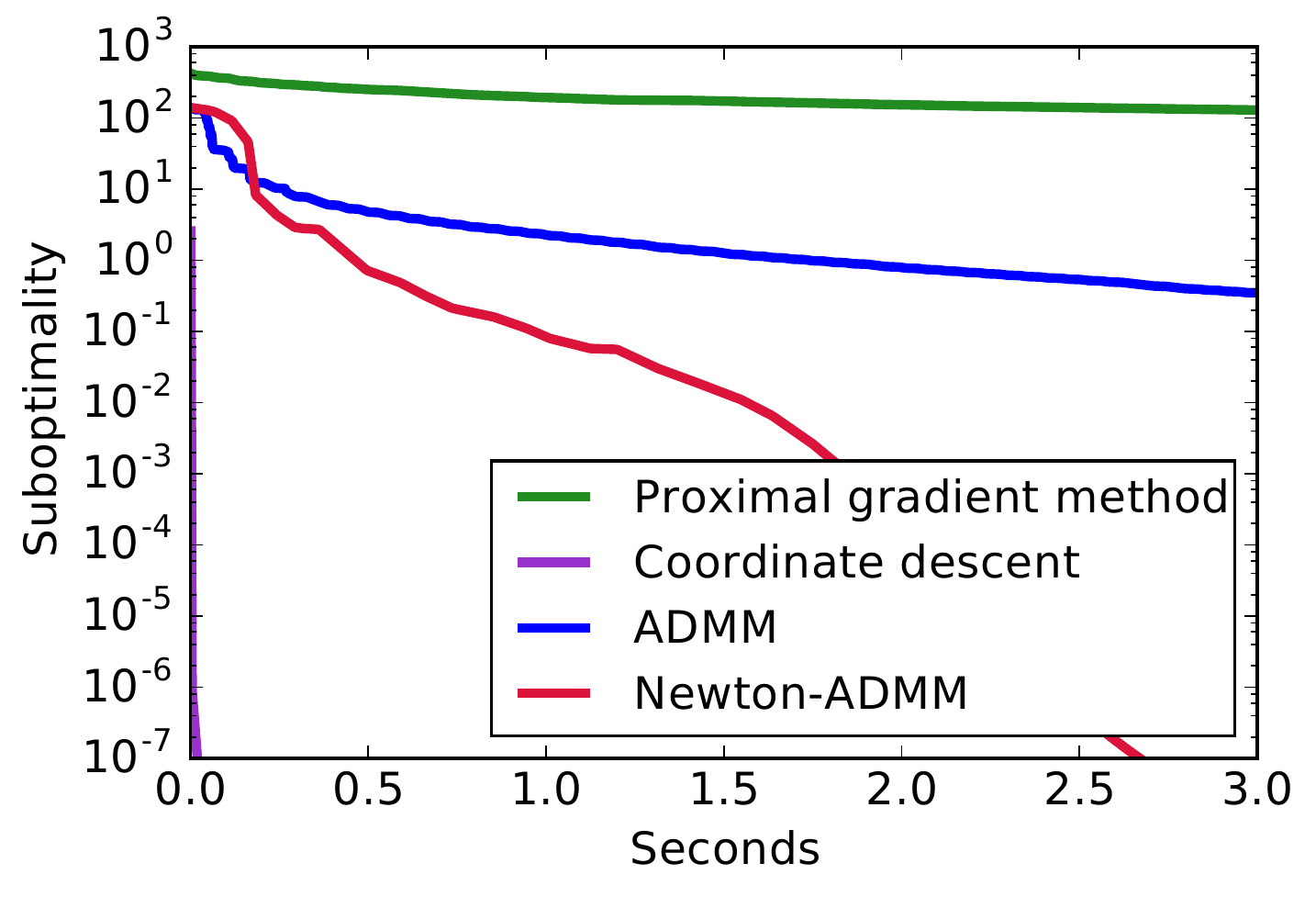}
\includegraphics[width=0.49\columnwidth]{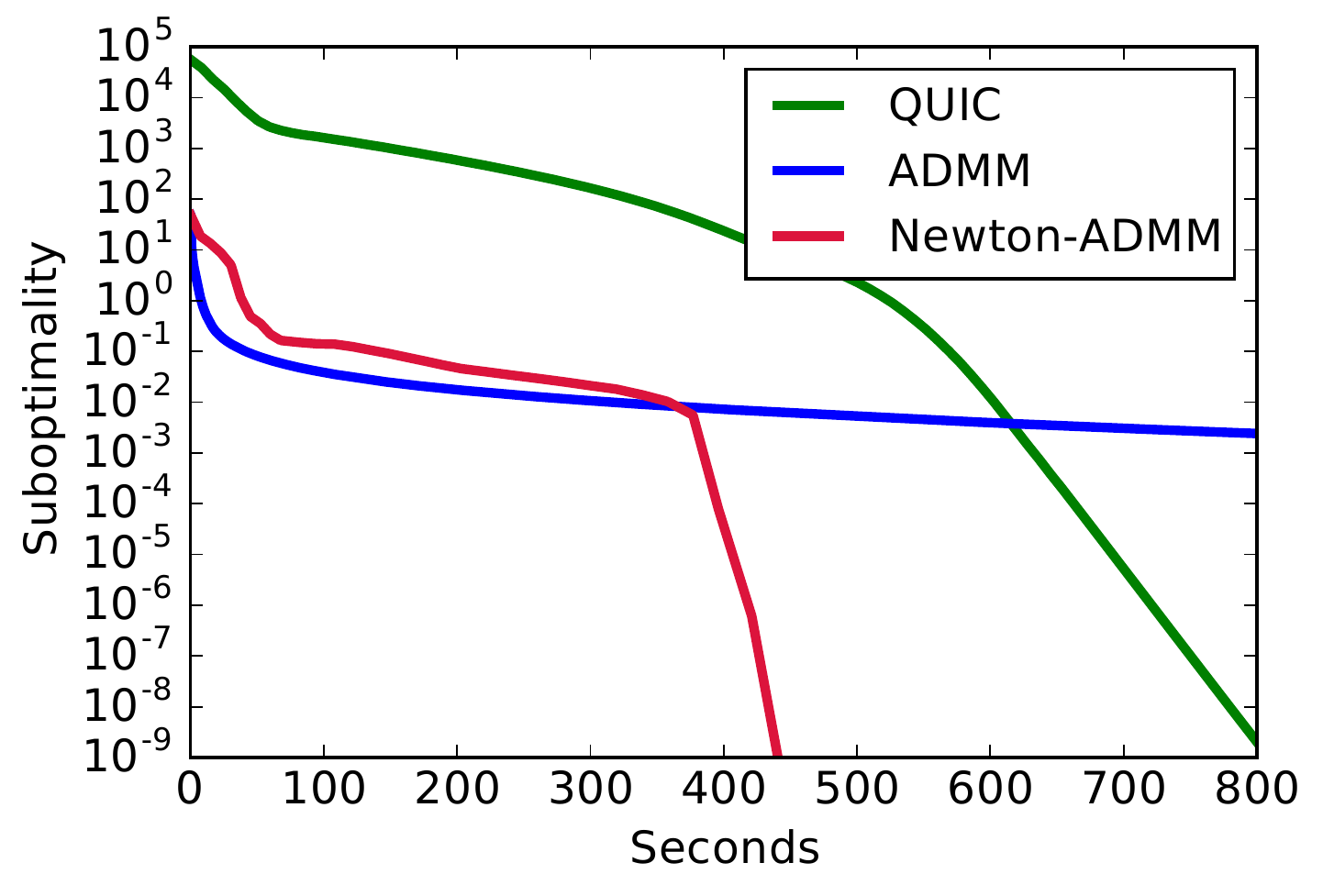}
\caption{Left: wallclock time vs.~log-distance to the optimal objective value, on the lasso problem, for the specialized Newton-ADMM method, standard ADMM, a proximal gradient method, and a heavily-optimized coordinate descent implementation (as a reference benchmark).  Right: for a sparse inverse covariance estimation problem, with specialized Newton-ADMM, standard ADMM, and QUIC \citep{hsieh2014quic}.  Best viewed in color.}
%\caption{Wallclock time vs.~log-distance to the optimal objective value, on the lasso problem, for the specialized Newton-ADMM method, standard ADMM, a proximal gradient method, and a heavily-optimized coordinate descent implementation (as a reference benchmark).  The specialized Newton-ADMM is competitive with these strong baselines.  Best viewed in color.}
\label{fig:lasso}
\end{center}
\vskip -0.2in
\end{figure} 

%\begin{figure}[ht]
%\vskip 0.2in
%\begin{center}
%\includegraphics[width=0.7\columnwidth]{./figs/spinvcov_baselines_ew.pdf}
%\caption{Wallclock time vs.~log-distance to the optimal objective value, on a sparse inverse covariance estimation problem, for specialized Newton-ADMM, standard ADMM, and QUIC \citep{hsieh2014quic}.  Newton-ADMM performs the best.  Best in color.}
%\label{fig:spinvcov}
%\end{center}
%\vskip -0.2in
%\end{figure} 

\section{Discussion}
\label{sec:disc}
We introduced Newton-ADMM, a new method for generic convex programming.  The basic idea is use a nonsmooth Newton method to find a fixed point of the residuals of the consecutive ADMM iterates generated by SCS, a state-of-the-art solver for cone programs; we showed that the basic idea is fairly general, and can be applied to accelerate (virtually) any ADMM-based algorithm.  We presented theoretical and empirical support that Newton-ADMM converges rapidly (\ie, quadratically) to a solution, outperforming SCS across several problems.

%\begin{small}
\paragraph{Acknowledgements.}
%\textbf{Acknowledgements.}
AA was supported by the DoE Computational
Science Graduate Fellowship DE-FG02-97ER25308.  EW was supported by DARPA, under
award number FA8750-17-2-0027.  We thank Po-Wei Wang and the referees for a
careful proof-reading. 
%\end{small}
%% Acknowledgements should only appear in the accepted version. 
%\section*{Acknowledgements} 
% 
%\textbf{Do not} include acknowledgements in the initial version of
%the paper submitted for blind review.
%
%If a paper is accepted, the final camera-ready version can (and
%probably should) include acknowledgements. In this case, please
%place such acknowledgements in an unnumbered section at the
%end of the paper. Typically, this will include thanks to reviewers
%who gave useful comments, to colleagues who contributed to the ideas, 
%and to funding agencies and corporate sponsors that provided financial 
%support.  

\clearpage            % Alnur put this.
\bibliography{refs}
\bibliographystyle{plainnat}

\newpage
\allowdisplaybreaks

\renewcommand\thefigure{S.\arabic{figure}}    
\renewcommand\thesection{S.\arabic{section}}
\renewcommand\theequation{S.\arabic{equation}}     
\setcounter{figure}{0}
\setcounter{section}{0}
\setcounter{equation}{0}

\title{Supplement to ``A Semismooth Newton Method for Fast, Generic Convex Programming''}
\maketitle

%This document contains proofs and supplementary details for the paper ``A Semismooth Newton Method for Fast, Generic Convex Programming''.  All section, equation, and figure numbers in this supplementary document are preceded by the letter S (all numbering without an S refers to the main paper).

\setcounter{footnote}{1}
\renewcommand*{\thefootnote}{\fnsymbol{footnote}}
\footnotetext{These authors contributed equally.}

\section{Proof of Lemma \ref{lem:projexpsemi}}
The proof relies on the proof of Lemma \ref{lem:jacprojexp}, below.  Let $z, \delta \in \reals^3$, and let $\delta \to 0$.  Suppose $z+\delta$ converges to a point that falls into one of the first three cases given in Section \ref{sec:bkgd}.  Then, from the statement and proof of Lemma \ref{lem:jacprojexp}, an element $J_{P_{\mathcal{K}_{\textrm{exp}}^*}}(z + \delta)$ of the generalized Jacobian of the projection onto the dual of the exponential cone at $z + \delta$, is just a matrix with fixed entries, since projections onto convex sets are continuous.  If $z + \delta$ converges to a point that falls into the fourth case, then brute force, \eg, using symbolic manipulation software, reveals that an element of the generalized Jacobian (\ie, the inverse of the specific 4x4 matrix $D$ given in \eqref{eq:24}, below) is also a constant matrix, even as $z_1^\star, z_2^\star, \nu^\star \to 0$; for completeness, we give $D^{-1}$ in \eqref{eq:25}, at the end of the supplement.  Thus in all the cases, the Jacobian is a constant matrix, which is enough to establish that the limit in \eqref{eq:semismooth} exists.  \hfill\qed
%the transpose of 

\section{Proof of Lemma \ref{lem:Fsemi}}
First, we give a useful result; its proof is elementary.
\begin{lemma} \label{lem:affsemi}
The affine transformation, $A F + b$, of a (strongly) semismooth map $F : \reals^k \to \reals^k$, with $A \in \reals^{k \times k}, \; b \in \reals^k$, is (strongly) semismooth.
\end{lemma}
\begin{proof}
First of all, we have that a map $F : \reals^k \to \reals^k$ is (strongly) semismooth if and only if its components $F_i$, for $i=1,\ldots,k$, are (strongly) semismooth \citep[Corollary 2.4]{qi1993nonsmooth}.  Additionally, we have that (strongly) semismooth maps are closed under linear combinations \citep[Proposition 1.75]{izmailov2014newton}.  Putting the two pieces together gives the claim.
\end{proof}

Now, from Lemma \ref{lem:projsemi}, we have that the projections onto the nonnegative orthant, second-order cone, positive semidefinite cone, as well as the free cone (an affine map, hence strongly semismooth \citep[Proposition 7.4.7]{facchinei2007finite}), are all strongly semismooth.  The map $F$, defined in \eqref{eq:F}, is just an affine transformation of these projections; thus, by \eqref{lem:affsemi}, it is strongly semismooth.

When $\mathcal{K}$, from \eqref{eq:coneprog2}, is the exponential cone, the analogous claim that the map $F$ is semismooth follows, from Lemma \ref{lem:projexpsemi}, in a similar way.  \hfill\qed

\section{Proof of Lemma \ref{lem:subg}}
\begin{proof}
We have that (i) the projection onto a convex set (\eg, the nonnegative orthant, second-order cone, positive semidefinite cone, exponential cone, and free cone), naturally, yields a convex set; (ii) the affine image of a convex set is a convex set; and (iii) retaining only some of the coordinates of a convex set is a convex set \citep[page 38]{boyd2004convex}.  Hence, the components $F_i$, for $i=1,\ldots,3k$, of the map $F : \reals^{3k} \to \reals^{3k}$, defined in \eqref{eq:F}, are convex functions.  Thus, by \citet[Proposition 1.2]{clarke1990optimization}, the $i$th row of any element of the generalized Jacobian is just a subgradient of $F_i$.  Now observe that the element $J$ of the generalized Jacobian, given in \eqref{eq:J}, is given by finding subgradients of the $F_i$.
\end{proof}

%\section{Further details on GMRES}

\section{Jacobian of the projection onto the second-order cone}
In Section \ref{sec:jacobs}, we stated that, in one case, the Jacobian of the projection onto the second-order cone at some point $z = (z_1, z_2) \in \reals^m$, with $z_1 \in \reals^{m-1}, \; z_2 \in \reals$, is a low-rank matrix $D \in \reals^{m \times m}$; the matrix $D$ is given by
\begin{equation}
D = 
\left[
\begin{array}{cccc}
\frac{1}{2} + \frac{z_2}{2 \| z_1 \|_2} - \frac{z_2}{2} \frac{(z_1)_1^2}{ \| z_1 \|_2^3 } & - \frac{z_2}{2} \frac{(z_1)_1 (z_1)_2}{ \| z_1 \|_2^3 } & \cdots & \frac{1}{2} \frac{(z_1)_1}{ \| z_1 \|_2 } \\
- \frac{z_2}{2} \frac{(z_1)_1 (z_1)_2}{ \| z_1 \|_2^3 } & \frac{1}{2} + \frac{z_2}{2 \| z_1 \|_2} - \frac{z_2}{2} \frac{(z_1)_2^2}{ \| z_1 \|_2^3 } & \cdots & \frac{1}{2} \frac{(z_1)_2}{ \| z_1 \|_2 } \\
\vdots & \vdots & \ddots & \vdots \\
\frac{1}{2} \frac{(z_1)_1}{ \| z_1 \|_2 } & \frac{1}{2} \frac{(z_1)_2}{ \| z_1 \|_2 } & \cdots & \frac{1}{2}
\end{array}
\right],
\end{equation}
which can be seen as the sum of diagonal and low-rank matrices.  Here, $(z_1)_i$ denotes the $i$th component of $z_1$.

\section{Proof of Lemma \ref{lem:jacprojpsd}}
Rewrite the projection onto the positive semidefinite cone as \eqref{eq:projpsd} as $P_{\mathcal{K}_{\textrm{psd}}}(Z) = Q \max(\Lambda, 0) Q^T$, where $Z = Q \max(\Lambda, 0) Q^T$ is the eigenvalue decomposition of some real, symmetric matrix $Z$, and the $\max$ here is interpreted diagonally.  Then, using the chain rule \citep{magnus1995matrix}, we get that
\begin{align*}
J_{P_{\mathcal{K}_{\textrm{psd}}}}(\vect Z) (d \vect Z) & = d \vect P_{\mathcal{K}_{\textrm{psd}}}(Z) \\
& = \vect \left( (d Q) \max(\Lambda, 0) Q^T + Q (d \max(\Lambda, 0)) Q^T + Q \max(\Lambda, 0) (d Q)^T \right);
\end{align*}
so, what remains is computing (each column of) $d Q$ and $d \max(\Lambda, 0)$, \ie, the differential of (each column of) the matrix of eigenvectors, and the differential of $\max(\Lambda, 0)$, respectively.  From \citet[Chapter 8]{magnus1995matrix}, we get that
\[
d Q_i = (\Lambda_{ii} I - Z)^+ (d Z) Q_i,
\]
where $Z^+$ denotes the pseudo-inverse of the matrix $Z$, and that
\[
\left[ d \max( \Lambda, 0 ) \right]_{ii} = I_+(\Lambda_{ii}) Q_i^T (d Z) Q_i,
\]
by applying the chain rule; here, $I_+(\cdot)$ is the indicator function of the nonnegative orthant, \ie, it equals 1 if its argument is nonnegative and 0 otherwise.  Replacing $d Z$ with some real, symmetric matrix $\tilde Z$ yields the claim.  \hfill\qed

\section{Further details on the per-iteration costs of SCS, Newton-ADMM, and CVXOPT}
Here, we elaborate on the costs of a single iteration of SCS, Newton-ADMM, and CVXOPT.  For simplicity, we consider the case where the cone $\mathcal{K}$, in the cone program \eqref{eq:coneprog}, is just a single cone (handling the case where $\mathcal{K}$ is the direct product of multiple cones is not hard); also, we are mostly interested in the high-dimensional case, where $n > m$.

During a single iteration of SCS, described in \eqref{eq:1} -- \eqref{eq:3}, we must carry out the computations outlined below:

\begin{itemize}
\item We must update the $\tilde u$ variable, which costs $O(\max\{n^2,m^2\})$ (see Section 4.1 of \citet{o2016conic}).

\item We must update the $u$ variable, the cost of which is dominated by the cost of projecting an $m$-vector onto the dual cone $\mathcal{K}^*$; for the case of projecting onto the positive semidefinite cone, we equivalently consider a matrix with dimensions $\sqrt{m} \times \sqrt{m}$.  These costs are as follows:

\begin{itemize}
\item For the nonnegative orthant, $\mathcal{K}_{\textrm{no}}$, the cost is $O(m)$.
\item For the second-order cone, $\mathcal{K}_{\textrm{soc}}$, the cost is $O(m)$.
\item For the positive semidefinite cone, $\mathcal{K}_{\textrm{psd}}$, the cost is $O(m^{3/2})$.
\item For the exponential cone, $\mathcal{K}_{\textrm{exp}}$, the cost is roughly $O(m^3)$.
\end{itemize}

\item We must update the $v$ variable, which has negligible cost.
\end{itemize}

Summing up, the cost of a single iteration of SCS is $O(\max\{n^2,m^2\})$ plus the cost of projecting onto the dual cone $\mathcal{K}^*$, as claimed in the main paper.

For Newton-ADMM, we must compute the ingredients on both sides of \eqref{eq:newtsys}, $F$ and $J$, as well as run GMRES and the backtracking line search.  Computing both $F$ and $J$ can be seen as essentially costing the same as a single iteration of SCS, \ie, the cost of projecting onto the dual cone $\mathcal{K}^*$ plus $O(\max\{n^2,m^2\})$; the backtracking line search, then, costs the number of backtracking iterations times the aforementioned cost.  Furthermore, running GMRES costs $O(\max\{n^2,m^2\})$, assuming it returns early.  Hence the cost of a single iteration of Newton-ADMM is (as claimed in the main paper) the number of backtracking iterations times the sum of two costs: the cost of projecting onto the dual cone $\mathcal{K}^*$ plus $O(\max\{n^2,m^2\})$.

Finally, turning to the interior-point method CVXOPT, it can be seen that the per-iteration cost here is dominated by solving the Newton system (1.11) in \citet{andersen2011interior}, essentially costing $O(n^3)$.

We mention that the above per-iteration costs can, of course, be improved by taking advantage of sparsity.

\section{Proof of Lemma \ref{lem:jacprojexp}}
First, from the Moreau decomposition given in \eqref{eq:moreau}, we get that
\[
J_{P_{\mathcal{K}_{\textrm{exp}}^*}}(z) = I - J_{P_{\mathcal{K}_{\textrm{exp}}}}(-z);
\]
so, what remains is to compute $J_{P_{\mathcal{K}_{\textrm{exp}}}}(z)$, for some $z \in \reals^m$.  Looking back at the first three cases given in Section \ref{sec:bkgd}, we get that
\begin{align*}
J_{P_{\mathcal{K}_{\textrm{exp}}}}(z) =
\begin{cases}
I, & z \in \mathcal{K}_{\textrm{exp}} \\
-I, & z \in \mathcal{K}_{\textrm{exp}}^* \\
\diag(1, I_+(z_2), I_+(z_3)), & z_1, z_2 < 0,
\end{cases}
\end{align*}
where $I_+(z_i), \; i=2,3$, is the indicator function of the nonnegative orthant, \ie, it equals 1 if $z_i \geq  0$ and 0 otherwise.  For the fourth case, the projection $P_{\mathcal{K}_{\textrm{exp}}}(z)$ is the solution to the optimization problem given in \eqref{eq:projexp}.  Now observe that (i) the optimization problem \eqref{eq:projexp} is, in fact, convex, since the constraint $\tilde z_2 > 0$ is really just implied by the domain of the function $\exp(\tilde z_1 / \tilde z_2)$; (ii) the optimization problem \eqref{eq:projexp} is feasible, since $z_1^\star = 1, \, z_2^\star = 1, \, z_3^\star = \exp(1)$ satisfies the constraint; and (iii) we can obtain a solution to the optimization problem \eqref{eq:projexp}, by using a Newton method \citep[Section 6.3.4]{parikh2014proximal}.

The rest of the proof relies on the KKT conditions for the optimization problem \eqref{eq:projexp}, as well as differentials (see, \eg, \citet{magnus1995matrix}).  The Lagrangian of the optimization problem \eqref{eq:projexp} is given by
\[
(1/2) \| \tilde z - z \|_2^2 + \nu ( \tilde z_2 \exp(\tilde z_1 / \tilde z_2) - \tilde z_3 ),
\]
where $\nu \in \reals$ is the dual variable.  Thus, we get that the KKT conditions for the optimization problem \eqref{eq:projexp}, at a solution $\gamma^\star = (z_1^\star, z_2^\star, z_3^\star, \nu^\star)$, are
\begin{align}
z_1^\star - z_1 + \nu^\star \exp(z_1^\star / z_2^\star) & = 0 \label{eq:20} \\
z_2^\star - z_2 + \nu^\star ( \exp(z_1^\star / z_2^\star) - (z_1^\star / z_2^\star) \exp(z_1^\star / z_2^\star) ) & = 0 \label{eq:21} \\
z_3^\star - z_3 - \nu^\star & = 0 \label{eq:22} \\
z_2^\star \exp(z_1^\star / z_2^\star) - z_3^\star & = 0. \label{eq:23}
\end{align}

Now consider the differentials $d z_1^\star, \, d z_2^\star, \, d z_3^\star, \, d z_4^\star$ and $d z_1, \, d z_2, \, d z_3, \, d z_4$ of the KKT conditions \eqref{eq:20} -- \eqref{eq:23}; we get for the condition \eqref{eq:20} that
\begin{align*}
d z_1^\star - d z_1 + (d \nu^\star) \exp(z_1^\star / z_2^\star) + \nu^\star (d \exp(z_1^\star / z_2^\star)) & = 0 \\
\iff d z_1^\star - d z_1 + (d \nu^\star) \exp(z_1^\star / z_2^\star) + \nu^\star \exp(z_1^\star / z_2^\star) (d(z_1^\star / z_2^\star)) & = 0 \\
\iff d z_1^\star - d z_1 + (d \nu^\star) \exp(z_1^\star / z_2^\star) + \nu^\star \exp(z_1^\star / z_2^\star) \left( \frac{d z_1^\star}{z_2^\star} - \frac{z_1^\star (d z_2^\star) }{(z_2^\star)^2} \right) & = 0 \\
\iff 
\left[
\begin{array}{cccc}
1 + \frac{\nu^\star \exp(z_1^\star / z_2^\star)}{z_2^\star}
& - \frac{\nu^\star \exp(z_1^\star / z_2^\star) z_1^\star}{(z_2^\star)^2}
& 0
& \exp(z_1^\star / z_2^\star)
\end{array}
\right]
\left[
\begin{array}{c}
d z_1^\star \\
d z_2^\star \\
d z_3^\star \\
d \nu^\star
\end{array}
\right]
& = 
d z_1.
\end{align*}

Repeating the above for the other conditions \eqref{eq:21} -- \eqref{eq:23}, we get that
\begin{align}
\underbrace{\left[
\begin{array}{cccc}
1 + \frac{\nu^\star \exp(z_1^\star / z_2^\star)}{z_2^\star} & - \frac{\nu^\star \exp(z_1^\star / z_2^\star) z_1^\star}{(z_2^\star)^2} & 0 & \exp(z_1^\star / z_2^\star) \\
- \frac{\nu^\star \exp(z_1^\star / z_2^\star) z_1^\star}{(z_2^\star)^2} & 1 + \frac{\nu^\star \exp(z_1^\star / z_2^\star) (z_1^\star)^2}{(z_2^\star)^3} & 0 & (1 - z_1^\star / z_2^\star) \exp(z_1^\star / z_2^\star) \\
0 & 0 & 1 & - 1 \\
\exp(z_1^\star / z_2^\star) & (1 - z_1^\star / z_2^\star) \exp(z_1^\star / z_2^\star) & -1 & 0
\end{array}
\right]}_{D}
\underbrace{\left[
\begin{array}{c}
d z_1^\star \\
d z_2^\star \\
d z_3^\star \\
d \nu^\star
\end{array}
\right]}_{d \gamma^\star}
= 
\underbrace{\left[
\begin{array}{c}
d z_1 \\
d z_2 \\
d z_3 \\
d \nu
\end{array}
\right]}_{d \gamma}, \label{eq:24}
\end{align}
\ie,
\[
D (d \gamma^\star) = d \gamma \quad \iff \quad d \gamma^\star = D^{-1} (d \gamma);
\]
here, $D$ is nonsingular, since the optimization problem \eqref{eq:projexp} is feasible.  So, by definition, the upper left 3x3 submatrix of $D^{-1}$ is the Jacobian of the projection onto the exponential cone, for the fourth case.  \hfill\qed
%-T

\section{Intuition behind some of the regularity conditions for Theorem \ref{thm:global1}, Theorem \ref{thm:global2}, and Theorem \ref{thm:local}}
Here, we elaborate on a couple of the regularity assumptions stated in the main paper.
%\begin{itemize}

\subsection{Regularity condition (A4)}
%\item[\textbf{A4.}] 
Roughly speaking, the condition (A4) can be seen as requiring that the directional derivative of $\tilde z \mapsto \| F(\tilde z) \|_2^2$ be bounded by $\alpha^{1/2} \| F(\tilde z) \|_2^2$.

We list some (useful) functions satisfying (A4):
\begin{itemize}
\item The function $F(z) = z^2$, for $z \in \reals$.  To show that the function $F$ satisfies (A4), we proceed by computing the required ingredients on both sides of (A4). Here, and for the rest of the section, we write $D_\Delta F^2(z)$ to mean the directional derivative of the function $F$ squared, in the direction $\Delta$, evaluated at $z$.

We compute, for $z > 0$ and the Newton direction $\Delta = -1$, the left-hand side of (A4),
\[
D_\Delta F^2(z) = -4z^3,
\]

and the right-hand side of (A4),
\[
-\alpha^{1/2} 2z^3.
\]

So, satisfying (A4) means
\[
-4z^3 \leq -\alpha^{1/2} 2z^3 \iff 2 \geq \alpha^{1/2},
\]
which is certainly true.  Repeating the argument for $z < 0$ and $\Delta = 1$ yields a similar result. (When $z = 0$, it is a solution.) Hence, $F(z) = z^2$ satisfies (A4).

\item The function $F(z) = \max(z+1, cz+1)$, with $z \in \reals$ and some $c > 0$.

We have, for the left-hand side of (A4):
\[
D_\Delta F^2(z = 0) = -2c.
\]

We have, for the right-hand side of (A4):
\[
\hat F(z = 0, \Delta = -1) = J(z = 0) \Delta = 1 \cdot (-1) = -1.
\]

So, satisfying (A4) means
\[
-2c \leq -\alpha^{1/2} \iff c \geq \alpha^{1/2}/2.
\]

In words, functions that satisfy (A4) cannot have $c$ too small.

\item An argument similar the one used above for $F(z) = z^2$ can be used to show that the function $F(z) = |z|$ also satisfies (A4).
\end{itemize}

We also establish, by using the condition (A4), that the backtracking line search, used in Algorithm \ref{alg:ours}, terminates.  Suppose, for contradiction, that the backtracking line search never terminates. Then, from the backtracking line search iteration described in Algorithm \ref{alg:ours}, we have, for all backtracking iterations $k$,
\[
( \|F(z) + \gamma^{(k)} \Delta\|_2^2 - \|F(z)\|_2^2 ) / \gamma^{(k)} \geq -\alpha \|F(z)\|_2^2.
\]

Taking the limit as $k \to \infty$, we get
\begin{equation}
D_\Delta \|F(z)\|_2^2 \geq -\alpha \|F(z)\|_2^2. \label{eq:27}
\end{equation}

On the other hand, expanding the right-hand side of (A4) gives
\begin{align}
\alpha^{1/2} F(z)^T \hat F(z, \Delta) & = \alpha^{1/2} \left( F(z)^T ( \hat F(z, \Delta) + F(z) ) - F(z)^T F(z) \right) \\
& \leq \alpha^{1/2} \left( \|F(z)\|_2 \|\hat F(z, \Delta) + F(z)\|_2 - \|F(z)\|_2^2 \right) \\
& \leq \alpha^{1/2} \left( \|F(z)\|_2 \varepsilon \|F(z)\|_2 - \|F(z)\|_2^2 \right) \\
& \leq \alpha^{1/2} \left( (1-\alpha^{1/2})\|F(z)\|_2^2 - \|F(z)\|_2^2 \right) \\
& = -\alpha^{1/4} \|F(z)\|_2^2. \label{eq:28}
\end{align}

Putting (A4) and \eqref{eq:28} above together immediately gives
\begin{equation}
D_\Delta \|F(z)\|_2^2 \leq -\alpha^{1/4} \|F(z)\|_2^2. \label{eq:29}
\end{equation}

But putting \eqref{eq:27} and \eqref{eq:29} together gives
\[
-\alpha \|F(z)\|_2^2 \leq D_\Delta \|F(z)\|_2^2 \leq -\alpha^{1/4} \|F(z)\|_2^2,
\]
a contradiction, since $\alpha \in (0,1)$.

\subsection{Regularity condition (A5)}
%\item[\textbf{A5.}] 
Roughly speaking, the condition (A5) says that the Newton step on each iteration cannot be too large.
%\end{itemize}

\section{Proof of Theorem \ref{thm:global1}}
\begin{proof}
We begin by recalling the condition under which backtracking line search continues, for a particular iteration of Newton's method; this happens as long as (see Algorithm \ref{alg:ours})
\begin{equation}
\| F(z^{(i)} + t^{(i)} \Delta^{(i)}) \|_2^2 \geq (1 - \alpha t^{(i)}) \| F(z^{(i)}) \|_2^2. \label{eq:btls}
\end{equation}
This means that when backtracking line search terminates, we get that
\begin{equation}
0 \leq \| F(z^{(i+1)}) \|_2^2 < (1 - \alpha t^{(i)}) \| F(z^{(i)}) \|_2^2 < \| F(z^{(i)}) \|_2^2. \label{eq:btls2}
\end{equation}
(To be clear, in order to get the second inequality here, we used the fact that backtracking line search terminates after \eqref{eq:btls} in Algorithm \ref{alg:ours} no longer holds.)  In order to get the third inequality here, we used the simple fact that $0 < 1 - \alpha t^{(i)} \leq 1$, since $0 < \alpha < 1/2$ and $0 < t^{(i)} \leq 1$.  So, we have shown that the sequence $(\| F(z^{(i)}) \|_2^2)_{i=1}^\infty$ is both bounded below and decreasing.  Note that this is just a sequence in $\reals$, and thus, by the monotone convergence theorem, it converges.  Furthermore, since every convergent sequence in $\reals$ is Cauchy, we get that
\begin{equation}
\lim_{i \to \infty} \left( \| F(z^{(i)}) \|_2^2 - \| F(z^{(i+1)}) \|_2^2 \right) = 0. \label{eq:9}
\end{equation}
On the other hand, by rearranging the second inequality in \eqref{eq:btls2}, we get that
\begin{equation}
\| F(z^{(i)}) \|_2^2 - \| F(z^{(i+1)}) \|_2^2 > \alpha t^{(i)} \| F(z^{(i)}) \|_2^2 \geq 0. \label{eq:13}
\end{equation}
So, \eqref{eq:9} along with taking the $\limsup_{i \to \infty}$ on both sides of \eqref{eq:13} yields that $\lim_{i \to \infty} \alpha t^{(i)} \| F(z^{(i)}) \|_2^2 = 0$.  But assumption (A1) says that $\limsup_{i \to \infty} t^{(i)} \to t > 0$, and since $\alpha > 0$, we get that $\lim_{i \to \infty} \tilde t \| F(z^{(i)}) \|_2^2 = 0$, for some $\tilde t > 0$, and so $\lim_{i \to \infty} \| F(z^{(i)}) \|_2^2 = 0$, which implies that $\lim_{i \to \infty} F(z^{(i)}) = 0$, as claimed.
\end{proof}

\section{Proof of Theorem \ref{thm:global2}}
\begin{proof}
First of all, by the assumption that $(z^{(i)})_{i=1}^\infty$ is convergent and assumption (A5), we must have that
\begin{equation}
0 \leq \| \Delta^{(i)} \|_2 \leq \frac{1}{C_2} \| \hat F(z^{(i)}, \Delta^{(i)}) \|_2 \leq \frac{\varepsilon+1}{C_2} \| F(z^{(i)}) \|_2, \label{eq:15}
\end{equation}
where the second inequality here follows by rearranging (A5), and the third inequality follows from \eqref{eq:trunc}, as well as the triangle inequality: after computing $\Delta^{(i)}$ on Newton iteration $i$, we are assured that
\begin{align*}
\| F(z^{(i)}) + \hat F(z^{(i)}, \Delta^{(i)}) \|_2 & \leq \varepsilon \| F(z^{(i)}) \|_2 \\
\implies \| \hat F(z^{(i)}, \Delta^{(i)}) \|_2 - \| F(z^{(i)}) \|_2 & \leq \varepsilon \| F(z^{(i)}) \|_2 \\
\iff \| \hat F(z^{(i)}, \Delta^{(i)}) \|_2 & \leq (\varepsilon + 1) \| F(z^{(i)}) \|_2.
\end{align*}
Hence, since
\[
\sup_{j, \ell} \dist(\Delta^{(j)}, \Delta^{(\ell)}) \quad \leq \quad \sup_{j} \| \Delta^{(j)} \|_2 + \sup_{\ell} \| \Delta^{(\ell)} \|_2,
\]
and because the right-hand side here is bounded (as per \eqref{eq:15}, as well as the fact that $(\| F(z^{(i)}) \|_2^2)_{i=1}^\infty$ is decreasing), we can conclude that the sequence $(\Delta^{(i)})_{i=1}^\infty$ is bounded.  (We used the Euclidean distance here.)

By the Bolzano-Weierstrass theorem (for Euclidean spaces), this sequence contains a convergent subsequence; let $(\Delta^{(i)})_{i \in \mathcal{S}}$, for some countable set $\mathcal{S}$, be this subsequence.  Define $\gamma^{(i)} = t^{(i)} / \beta$, \ie, $\gamma^{(i)}$ is the last $t^{(i)}$ for which \eqref{eq:btls} was actually true (\ie, when checked at the start of the $(i+1)$th Newton iteration).  Then we get
\[
\| F(z^{(i)} + \gamma^{(i)} \Delta^{(i)}) \|_2^2 - \| F(z^{(i)}) \|_2^2 \geq - \alpha \gamma^{(i)} \| F(z^{(i)}) \|_2^2;
\]
dividing through by $\gamma^{(i)}$ and taking limits gives (observe that, from assumption (A2), $\limsup_{i \to \infty} t^{(i)} = 0 \implies \lim_{i \to \infty} t^{(i)} = 0$)
\begin{align}
- \alpha \| F(z) \|_2^2 & \leq \lim_{\substack{i,j \to \infty, \; j \in \mathcal{S}}} \frac{\| F(z^{(i)} + \gamma^{(i)} \Delta^{(j)}) \|_2^2 - \| F(z^{(i)}) \|_2^2}{\gamma^{(i)}} \label{eq:12} \\
& \leq \lim_{\substack{i,j \to \infty, \; j \in \mathcal{S}}} \alpha^{1/2} F(z^{(i)})^T \hat F(z^{(i)}, \Delta^{(j)}), \label{eq:10}
\end{align}
with the second line here following by assumption (A4).  Expanding the right-hand side of \eqref{eq:10}, we get
\begin{align*}
\alpha^{1/2} F(z^{(i)})^T \hat F(z^{(i)}, \Delta^{(j)}) & = \alpha^{1/2} F(z^{(i)})^T \left( \hat F(z^{(i)}, \Delta^{(j)}) + F(z^{(i)}) \right) - \alpha^{1/2} F(z^{(i)})^T F(z^{(i)}) \notag \\
& \leq \alpha^{1/2} \| F(z^{(i)}) \|_2 \| F(z^{(i)}) + \hat F(z^{(i)}, \Delta^{(j)}) \|_2 - \alpha^{1/2} \| F(z^{(i)}) \|_2^2 \notag \\
& \leq \alpha^{1/2} \varepsilon \| F(z^{(i)}) \|_2^2 - \alpha^{1/2} \| F(z^{(i)}) \|_2^2 \notag \\
& = - \alpha^{1/2} \| F(z^{(i)}) \|_2^2 (1 - \varepsilon), % \label{eq:11}
\end{align*}
with the second line following from the Cauchy-Schwarz inequality, and the third from \eqref{eq:trunc}.  So, we obtain for the right-hand side of \eqref{eq:10} that
\begin{align}
\lim_{\substack{i,j \to \infty, \; j \in \mathcal{S}}} \alpha^{1/2} F(z^{(i)})^T \hat F(z^{(i)}, \Delta^{(j)}) \leq - \alpha^{1/2} \| F(z) \|_2^2 (1 - \varepsilon). \label{eq:11}
\end{align}

Putting together \eqref{eq:12} and \eqref{eq:11}, we get that
\begin{align*}
- \alpha \| F(z) \|_2^2 \leq - \alpha^{1/2} \| F(z) \|_2^2 (1 - \varepsilon) \quad \iff \quad 0 \geq \alpha^{1/2} \| F(z) \|_2^2 \left( (1 - \varepsilon) - \alpha^{1/2} \right).
\end{align*}

Now, by assumption (A3), we require that $\varepsilon < 1 - \alpha^{1/2} \iff (1 - \varepsilon) - \alpha^{1/2} > 0$; thus, we must have that $\| F(z) \|_2^2 = 0 \iff F(z) = 0$, as claimed.
\end{proof}

\section{Proof of Theorem \ref{thm:local}}
\newcommand{\resid}{\mathop{\bf res}}
\begin{proof}
The theorem establishes that the iterates $(z^{(i)})_{i=1}^\infty$ generated by Algorithm \ref{alg:ours} are \textit{locally quadratically convergent}, \ie, we get, for large enough $i$ and some $C > 0$, that
\[
\lim_{i \to \infty} \frac{ |z^{(i+1)} - z| }{ (z^{(i)} - z)^2 } = C.
\]

Let $\resid(i) = F(z^{(i)}) + J(z^{(i)}) \Delta^{(i)}$, for convenience.  We begin by making two useful observations.

First, using the second part of assumption (A6), we get that
\begin{align}
\| F(z^{(i)}) - \resid(i) \|_2 & = \| J(z^{(i)}) \Delta^{(i)} \|_2 \notag \\
& \leq \| J(z^{(i)}) \|_2 \| \Delta^{(i)} \|_2 \notag \\
& \leq C_3 \| \Delta^{(i)} \|_2. \label{eq:17}
\end{align}

On the other hand, using the triangle inequality as well as \eqref{eq:trunc}, we get that 
\begin{align}
\| F(z^{(i)}) - \resid(i) \|_2 & \geq \| F(z^{(i)}) \|_2 - \| \resid(i) \|_2 \notag \\
%& \geq \| F(z^{(i)}) \|_2 - O( \| F(z^{(i)}) \|_2^2 ) \notag \\
& \geq \| F(z^{(i)}) \|_2 - \varepsilon \| F(z^{(i)}) \|_2 \notag \\
& \geq (1 - \varepsilon) \| F(z^{(i)}) \|_2. \label{eq:18}
\end{align}
%for some constant $C_4 > 0$, and for large enough $i$.

So, putting together \eqref{eq:17} and \eqref{eq:18}, we get that
\[
(1 - \varepsilon) \| F(z^{(i)}) \|_2 \leq C_3 \| \Delta^{(i)} \|_2 \quad \implies \quad \| F(z^{(i)}) \|_2 \leq C_4 \| \Delta^{(i)} \|_2,
\]
for some constant $C_4 > 0$, since $1 - \varepsilon > 0$.  Squaring both sides, it follows that
\begin{align}
\| F(z^{(i)}) \|_2 & \leq C_4 \| \Delta^{(i)} \|_2 \notag \\
\implies \| F(z^{(i)}) \|_2^2 & \leq C_4^2 \| \Delta^{(i)} \|_2^2 \notag \\
\implies \| \resid(i) \|_2 & \leq C_5 \| \Delta^{(i)} \|_2^2 \notag \\
\implies \frac{ \| \resid(i) \|_2 }{ \| \Delta^{(i)} \|_2^2 } & \leq C_5, \label{eq:19}
\end{align}
where $C_5 > 0$ is some constant, and the third line follows because \eqref{eq:trunc} and assumption (A3) tell us that $\| \resid(i) \|_2 \leq C \| F(z^{(i)}) \|_2^2$ for some constant $C > 0$.  Finally, \citet[Theorem 2.5]{facchinei1997nonsmooth} and the second part of assumption (A6) tell us that the sequence of iterates $(z^{(i)})_{i=1}^\infty \to z$ converges quadratically, with $F(z) = 0$, as claimed.
\end{proof}

\section{Further details on the minimum variance portfolio optimization example}
Here, we elaborate on putting the minimum variance portfolio optimization problem \eqref{eq:portopt} into the cone form of \eqref{eq:coneprog}.

First, we rewrite the minimum variance portfolio optimization problem \eqref{eq:portopt} as
\begin{equation*}
\begin{array}{ll}
\minimizewrt{\theta \in \mathbf{R}^p, \, w \in \mathbf{R}} & w \\
\subjectto & 
\left\|
\left[
\begin{array}{c}
2 \Sigma^{1/2} \theta \\
1 - w
\end{array}
\right]
\right\|_2
\leq 1 + w
\\
& 1 \leq \ones^T \theta \leq 1,
\end{array}
\end{equation*}
where we used the simple fact \citep[Equation 8]{lobo1998applications} that
\[
\alpha^T \alpha \leq \gamma \delta \quad \iff \quad 
\left\|
\left[
\begin{array}{c}
2 \alpha \\
\gamma - \delta
\end{array}
\right]
\right\|_2
\leq \gamma + \delta,
\]
for some vector $\alpha$ and nonnegative constants $\theta, \gamma$ (for us, $\alpha = \Sigma^{1/2} \theta$, $\gamma = 1$, and $\delta = w$).  Then, we rewrite the above problem as
\begin{equation*}
\begin{array}{ll}
\minimizewrt{x \in \mathbf{R}^{p+1}} & c^T x \\
\subjectto & \| G_1 x + h \|_2 \leq q^T x + z \\
& G_2 x \leq 1, \quad G_3 x \leq -1,
\end{array}
\end{equation*}
where we defined
\begin{align*}
x & = 
\left[
\begin{array}{c}
\theta \\
w
\end{array}
\right] \\\
c & = 
\left[
\begin{array}{c}
0 \\
1
\end{array}
\right] \\
G_1 & = 
\left[
\begin{array}{cc}
2 \Sigma^{1/2} & 0 \\
0 & -1
\end{array}
\right], \quad h = 
\left[
\begin{array}{c}
0 \\
1
\end{array}
\right] \\
q & = 
\left[
\begin{array}{c}
0 \\
1
\end{array}
\right], \quad z = 1 \\
G_2 & = 
\left[
\begin{array}{ccc}
\ones^T & 0
\end{array}
\right], \quad
G_3 = 
\left[
\begin{array}{ccc}
-\ones^T & 0
\end{array}
\right].
\end{align*}

Finally, we just use
\begin{align*}
A = 
\left[
\begin{array}{c}
-G_1 \\
-q^T \\
G_2 \\
G_3
\end{array}
\right], \quad b = 
\left[
\begin{array}{c}
h \\
z \\
1 \\
-1
\end{array}
\right], \quad \mathcal{K} = \mathcal{K}_{\textrm{soc}}^{p+2} \times \mathcal{K}_{\textrm{no}} \times \mathcal{K}_{\textrm{no}},
\end{align*}
to get the cone form of \eqref{eq:coneprog}; here, $\mathcal{K}_{\textrm{soc}}^{p+2}$ denotes the $(p+1)$-dimensional second-order cone.

\section{Further details on the $\ell_1$-penalized logistic regression example}
Here, we elaborate on putting the $\ell_1$-penalized logistic regression problem \eqref{eq:logreg} into the cone form of \eqref{eq:coneprog}.  To keep the notation light, we write
\[
z_i = y_i X_{i \cdot} \theta.
\]

Now, for $i=1,\ldots,N$, we use the simple fact \citep[Section 9.4.1]{serrano2015algorithms} that
\[
\log \left( \sum_i \exp(\alpha_i) \right) \leq -\theta \quad \iff \quad \sum_i \exp(\alpha_i + \theta) \leq 1,
\]
for $\alpha_i, \theta \in \reals$, in order to conclude that
\begin{equation}
\log(\exp(0) + \exp(z_i)) \leq w_i \quad \iff \quad \exp(-w_i) + \exp(z_i - w_i) \leq 1, \label{eq:7}
\end{equation}
where the $w_i \in \reals$ are some variables that we will introduce, later on.  Next, we ``split'' the right-hand side of \eqref{eq:7} into the following set of constraints:
\begin{align*}
\exp(-w_i) & \leq \ell_i \quad \iff \quad
\left[
\begin{array}{c}
-w_i \\
1 \\
\ell_i
\end{array}
\right]
\in \mathcal{K}_{\exp}, \quad i=1,\ldots,N, \\
\exp(z_i - w_i) &\leq q_i \quad \iff \quad
\left[
\begin{array}{c}
z_i - w_i \\
1 \\
q_i
\end{array}
\right]
\in \mathcal{K}_{\exp}, \quad i=1,\ldots,N,
\\
\ell_i + q_i & \leq 1, \quad i=1,\ldots,N,
\end{align*}
where $\ell_i, q_i \in \reals$ are more new variables.  Thus, we can write the $\ell_1$-penalized logistic regression problem \eqref{eq:logreg} as
\begin{equation*}
\begin{array}{ll}
\minimizewrt{\substack{\theta \in \mathbf{R}^p, \, w \in \mathbf{R}^N, \\ t \in \mathbf{R}^p, \, \ell \in \mathbf{R}^N, \\ q \in \mathbf{R}^N}} & \ones^T w + \lambda \ones^T t \\
\subjectto & 
\left[
\begin{array}{c}
-w_i \\
1 \\
\ell_i
\end{array}
\right]
\in \mathcal{K}_{\exp}, \quad i=1,\ldots,N \\
& \left[
\begin{array}{c}
y_i X_{i \cdot} \theta - w_i \\
1 \\
q_i
\end{array}
\right]
\in \mathcal{K}_{\exp}, \quad i=1,\ldots,N \\
& \ell + q \leq \ones \\
& - t \leq \theta \leq t.
\end{array}
\end{equation*}

Finally, to get the cone form of \eqref{eq:coneprog}, we use
\begin{align*}
x & = 
\left[
\begin{array}{c}
\theta \\
w \\
t \\
\ell \\
q
\end{array}
\right], \\
c & = 
\left[
\begin{array}{c}
0 \\
\ones \\
\lambda \ones \\
0 \\
0
\end{array}
\right], \\
A & = 
\left[
\begin{array}{ccccc}
 & & G_1 & & \\
 & & \vdots & & \\
 & & G_N & & \\
 & & H_1 & & \\
 & & \vdots & & \\
 & & H_N & & \\
0 & 0 & 0 & I & I \\
-I & 0 & -I & 0 & 0 \\
I & 0 & -I & 0 & 0
\end{array}
\right], \\
G_i & =
\left[
\begin{array}{ccccc}
0 & e_i^T & 0 & 0 & 0 \\
0 & 0 & 0 & 0 & 0 \\
0 & 0 & 0 & -e_i^T & 0
\end{array}
\right], \quad
H_i =
\left[
\begin{array}{ccccc}
- y_i X_{i \cdot} & e_i^T & 0 & 0 & 0 \\
0 & 0 & 0 & 0 & 0 \\
0 & 0 & 0 & 0 & -e_i^T
\end{array}
\right], \quad i=1,\ldots,N, \\
b & =
\left[
\begin{array}{c}
h \\
\vdots \\ 
h \\
h \\
\vdots \\ 
h \\
\ones \\
0 \\
0
\end{array}
\right], \quad
h = 
\left[
\begin{array}{c}
0 \\
1 \\
0
\end{array}
\right], \\
\mathcal{K} & = \underbrace{\mathcal{K}_{\exp} \times \dots \times \mathcal{K}_{\exp}}_{N} \times \underbrace{\mathcal{K}_{\exp} \times \dots \times \mathcal{K}_{\exp}}_{N} \times \mathcal{K}_{\textrm{no}}^N \times \mathcal{K}_{\textrm{no}}^p \times \mathcal{K}_{\textrm{no}}^p;
\end{align*}
here, $e_i, \; i=1,\ldots,N$ denotes the $i$th standard basis vector in $\reals^N$, and $\mathcal{K}_{\textrm{no}}^i$ denotes the $i$-dimensional nonnegative orthant.

\section{Further details on the robust PCA example}
Here, we elaborate on putting the robust PCA problem \eqref{eq:robust} into the cone form of \eqref{eq:coneprog}.

First, we observe that, using duality arguments (see, \eg, \citet[Section 3]{fazel2001rank} or \citet[Proposition 2.1]{recht2010guaranteed}), we can rewrite the robust PCA problem \eqref{eq:robust} as
\begin{equation*}
\begin{array}{ll}
\minimizewrt{\substack{W_1 \in \mathbf{R}^{N \times N}, W_2 \in \mathbf{R}^{p \times p}, \\ t \in \mathbf{R}^{Np}, \, L, S \in \mathbf{R}^{N \times p}}} & (1/2)(\tr(W_1) + \tr(W_2)) \\
\subjectto & -t \leq \vect(S) \leq t \\
& \ones^T t \leq \lambda \\
& L + S = X \\
& \left[
\begin{array}{cc}
W_1 & L \\
L^T & W_2
\end{array}
\right] \succeq 0.
\end{array}
\end{equation*}

To get the cone form of \eqref{eq:coneprog}, we use
\begin{align*}
x & =
\left[
\begin{array}{c}
\vect(W_1) \\
\vect(W_2) \\
t \\
\vect(L) \\
\vect(S)
\end{array}
\right], \\
c & = 
\left[
\begin{array}{c}
(1/2) \vect(I) \\
(1/2) \vect(I) \\
0 \\
0 \\
0
\end{array}
\right], \\
A & = 
\left[
\begin{array}{ccccc}
0 & 0 & -I & 0 & -I \\
0 & 0 & -I & 0 & I \\
0 & 0 & \ones^T & 0 & 0 \\
0 & 0 & 0 & I & I \\
0 & 0 & 0 & -I & -I \\
G_{W_1} & G_{W_2} & 0 & G_{L} & 0
\end{array}
\right], \\
G_{W_1} & = 
\left[
\begin{array}{ccccc}
\vect(G_{W_1}^{(1,1)}) & \vect(G_{W_1}^{(2,1)}) & \cdots & \vect(G_{W_1}^{(N-1,N)}) & \vect(G_{W_1}^{(N,N)})
\end{array}
\right], \\
& \quad\quad \textrm{where $G_{W_1}^{(i,j)}$ is $0$ except with the $(i,j)$th entry of its upper left $N \times N$ block set to 1}, \\
G_{W_2} & = 
\left[
\begin{array}{ccccc}
\vect(G_{W_2}^{(1,1)}) & \vect(G_{W_2}^{(2,1)}) & \cdots & \vect(G_{W_2}^{(p-1,p)}) & \vect(G_{W_2}^{(p,p)})
\end{array}
\right], \\
& \quad\quad \textrm{where $G_{W_2}^{(i,j)}$ is $0$ except with the $(i,j)$th entry of its bottom right $p \times p$ block set to 1}, \\
G_{L} & = 
\left[
\begin{array}{ccccc}
\vect(G_{L}^{(1,1)}) & \vect(G_{L}^{(2,1)}) & \cdots & \vect(G_{L}^{(N-1,p)}) & \vect(G_{L}^{(N,p)})
\end{array}
\right], \\
& \quad\quad \textrm{where $G_{L}^{(i,j)}$ is $0$ except with the $(i,j)$th entry of its upper right $N \times p$ block} \\
& \quad\quad \textrm{and the $(j,i)$th entry of its lower left $p \times N$ block set to 1}, \\
b & = 
\left[
\begin{array}{c}
0 \\
0 \\
\lambda \\
\vect(X) \\
-\vect(X) \\
0
\end{array}
\right], \\
\mathcal{K} & = \mathcal{K}_{\textrm{no}}^{Np} \times \mathcal{K}_{\textrm{no}}^{Np} \times \mathcal{K}_{\textrm{no}} \times \mathcal{K}_{\textrm{no}}^{Np} \times \mathcal{K}_{\textrm{no}}^{Np} \times \mathcal{K}_{\textrm{psd}}^{N+p}.
\end{align*}
Here, $\mathcal{K}_{\textrm{psd}}^i$ denotes the $(i \times i)$-dimensional positive semidefinite cone.  Also, observe that the last row of $A, b$ above encodes the constraint
\[
\left[
\begin{array}{cc}
W_1 & L \\
L^T & W_2
\end{array}
\right] \in \mathcal{K}_{\textrm{psd}}^{N+p},
\]
which we can write as a linear matrix inequality \citep[Equation 1.7]{andersen2011interior}:
%\citep[Section 4.1.1]{el2000advances}
\begin{align*}
\left[
\begin{array}{cc}
W_1 & L \\
L^T & W_2
\end{array}
\right] \succeq 0 
& \iff 
\sum_{i,j} G_{W_1}^{(i,j)} (W_1)_{ij} + \sum_{i,j} G_{W_2}^{(i,j)} (W_2)_{ij} + \sum_{i,j} G_{L}^{(i,j)} L_{ij} \succeq 0 \\
& \iff 
\left[
\begin{array}{ccccc}
G_{W_1} & G_{W_2} & 0 & G_{L} & 0
\end{array}
\right]
x \succeq 0.
\end{align*}

\newcommand\scalemath[2]{\scalebox{#1}{\mbox{\ensuremath{\displaystyle #2}}}}
%\clearpage
\begin{adjustwidth*}{-5.75cm}{-2cm}
\begin{landscape}
\thispagestyle{empty}
\vspace*{\fill}
Expression for the matrix $D^{-1}$, used in the proof of Lemma \ref{lem:projexpsemi}:
%\[
\begin{equation}
\scalemath{0.6}{
D^{-1} = (1/k) \cdot %auto-ignore
\left[
\begin{array}{cccc}
-1-\frac{e^{z_1^\star/z_2^\star} \nu^\star (z_1^\star)^2}{(z_2^\star)^3}-\frac{e^{z_1^\star/z_2^\star} e^{ (z_1^\star)^3}}{(z_2^\star)^3}+\frac{2
e^{z_1^\star/z_2^\star} e^{ (z_1^\star)^2}}{(z_2^\star)^2}-\frac{e^{z_1^\star/z_2^\star} e^{ z_1^\star}}{z_2^\star} & -\frac{e^{z_1^\star/z_2^\star} \nu^\star z_1^\star}{(z_2^\star)^2}-\frac{e^{z_1^\star/z_2^\star}
e^{ (z_1^\star)^2}}{(z_2^\star)^2}+\frac{e^{z_1^\star/z_2^\star} e^{ z_1^\star}}{z_2^\star} & -e^{z_1^\star/z_2^\star}-\frac{e^{\frac{2 z_1^\star}{z_2^\star}} \nu^\star z_1^\star}{(z_2^\star)^2} & -e^{z_1^\star/z_2^\star}-\frac{e^{\frac{2 z_1^\star}{z_2^\star}}
\nu^\star z_1^\star}{(z_2^\star)^2} \\
e^{\frac{2 z_1^\star}{z_2^\star}}-\frac{e^{z_1^\star/z_2^\star} \nu^\star z_1^\star}{(z_2^\star)^2}-\frac{e^{\frac{2
z_1^\star}{z_2^\star}} z_1^\star}{z_2^\star} & -1-e^{\frac{2 z_1^\star}{z_2^\star}}-\frac{e^{z_1^\star/z_2^\star} \nu^\star}{z_2^\star} & -e^{z_1^\star/z_2^\star}-\frac{e^{\frac{2 z_1^\star}{z_2^\star}}
\nu^\star}{z_2^\star}+\frac{e^{z_1^\star/z_2^\star} z_1^\star}{z_2^\star} & -e^{z_1^\star/z_2^\star}-\frac{e^{\frac{2 z_1^\star}{z_2^\star}} \nu^\star}{z_2^\star}+\frac{e^{z_1^\star/z_2^\star}
z_1^\star}{z_2^\star} \\
-e^{z_1^\star/z_2^\star}+\frac{e^{z_1^\star/z_2^\star} e^{ \nu^\star (z_1^\star)^3}}{(z_2^\star)^4}-\frac{e^{\frac{2
z_1^\star}{z_2^\star}} \nu^\star (z_1^\star)^2}{(z_2^\star)^3}-\frac{e^{z_1^\star/z_2^\star} e^{ \nu^\star (z_1^\star)^2}}{(z_2^\star)^3} & \frac{e^{z_1^\star/z_2^\star} e^{ \nu^\star
(z_1^\star)^2}}{(z_2^\star)^3}-\frac{e^{\frac{2 z_1^\star}{z_2^\star}} \nu^\star z_1^\star}{(z_2^\star)^2}-\frac{e^{z_1^\star/z_2^\star} e^{ \nu^\star
z_1^\star}}{(z_2^\star)^2}+\frac{e^{ (z_1^\star)^2}}{(z_2^\star)^2}-\frac{e^{ z_1^\star}}{z_2^\star} & -e^{\frac{2 z_1^\star}{z_2^\star}}+\frac{e^{\frac{2 z_1^\star}{z_2^\star}} e^{ \nu^\star (z_1^\star)^2}}{(z_2^\star)^3}-\frac{e^{z_1^\star/z_2^\star}
e^{ (z_1^\star)^3}}{(z_2^\star)^3}-\frac{e^{\frac{3 z_1^\star}{z_2^\star}} \nu^\star z_1^\star}{(z_2^\star)^2}-\frac{e^{\frac{2 z_1^\star}{z_2^\star}}
e^{ \nu^\star z_1^\star}}{(z_2^\star)^2}+\frac{2 e^{z_1^\star/z_2^\star} e^{ (z_1^\star)^2}}{(z_2^\star)^2}-\frac{e^{z_1^\star/z_2^\star} e^{
z_1^\star}}{z_2^\star} & 1+\frac{e^{z_1^\star/z_2^\star} \nu^\star (z_1^\star)^2}{(z_2^\star)^3}+\frac{e^{z_1^\star/z_2^\star} \nu^\star}{z_2^\star} \\
-e^{z_1^\star/z_2^\star}+\frac{e^{z_1^\star/z_2^\star}
e^{ \nu^\star (z_1^\star)^3}}{(z_2^\star)^4}-\frac{e^{\frac{2 z_1^\star}{z_2^\star}} \nu^\star (z_1^\star)^2}{(z_2^\star)^3}-\frac{e^{z_1^\star/z_2^\star}
e^{ \nu^\star (z_1^\star)^2}}{(z_2^\star)^3} & \frac{e^{z_1^\star/z_2^\star} e^{ \nu^\star (z_1^\star)^2}}{(z_2^\star)^3}-\frac{e^{\frac{2 z_1^\star}{z_2^\star}}
\nu^\star z_1^\star}{(z_2^\star)^2}-\frac{e^{z_1^\star/z_2^\star} e^{ \nu^\star z_1^\star}}{(z_2^\star)^2}+\frac{e^{ (z_1^\star)^2}}{(z_2^\star)^2}-\frac{e^{
z_1^\star}}{z_2^\star} & 1+\frac{e^{z_1^\star/z_2^\star} \nu^\star (z_1^\star)^2}{(z_2^\star)^3}+\frac{e^{z_1^\star/z_2^\star} \nu^\star}{z_2^\star} & 1+\frac{e^{z_1^\star/z_2^\star} \nu^\star (z_1^\star)^2}{(z_2^\star)^3}+\frac{e^{z_1^\star/z_2^\star}
\nu^\star}{z_2^\star}
\end{array}
\right]
,
}
\label{eq:25}
\end{equation}
%\]
\[
\textrm{where } k = %auto-ignore
-1-e^{\frac{2 z_1^\star}{\hat z_2}}-\frac{e^{z_1^\star/z_2^\star} \nu^\star (z_1^\star)^2}{(z_2^\star)^3}+\frac{e^{\frac{2 z_1^\star}{z_2^\star}} e^{ \nu^\star (z_1^\star)^2}}{(z_2^\star)^3}-\frac{e^{z_1^\star/z_2^\star} e^{ (z_1^\star)^3}}{(z_2^\star)^3}-\frac{e^{\frac{3 z_1^\star}{z_2^\star}} \nu^\star z_1^\star}{(z_2^\star)^2}-\frac{e^{\frac{2 z_1^\star}{z_2^\star}} e^{ \nu^\star z_1^\star}}{(z_2^\star)^2}+\frac{2 e^{z_1^\star/z_2^\star} e^{ (z_1^\star)^2}}{(z_2^\star)^2}-\frac{e^{z_1^\star/z_2^\star} \nu^\star}{z_2^\star}-\frac{e^{z_1^\star/z_2^\star} e^{ z_1^\star}}{z_2^\star}
.
\]
\vspace*{\fill}
\end{landscape}
\end{adjustwidth*}
%\clearpage

%% Acknowledgements should only appear in the accepted version. 
%\section*{Acknowledgements} 
% 
%\textbf{Do not} include acknowledgements in the initial version of
%the paper submitted for blind review.
%
%If a paper is accepted, the final camera-ready version can (and
%probably should) include acknowledgements. In this case, please
%place such acknowledgements in an unnumbered section at the
%end of the paper. Typically, this will include thanks to reviewers
%who gave useful comments, to colleagues who contributed to the ideas, 
%and to funding agencies and corporate sponsors that provided financial 
%support.  

%\clearpage            % Alnur put this.
%\bibliography{refs}

\end{document}